\documentclass[12pt, reqno]{amsart}
\usepackage{amsmath, amsthm, amscd, amsfonts, amssymb, graphicx, color}

\theoremstyle{plain}
\newtheorem{theorem}{Theorem}[section]
\newtheorem{lemma}[theorem]{Lemma}
\newtheorem{corollary}[theorem]{Corollary}
\newtheorem{Prop}[theorem]{Proposition}

 \theoremstyle{definition}
\newtheorem{definition}[theorem]{Definition}

\newtheorem{Rem}[theorem]{Remark}
\newtheorem{?}[theorem]{Problem}
\newtheorem{Ex}[theorem]{Example}
\newtheorem{Exs}[theorem]{Examples}

\usepackage{hyperref}
\hypersetup{colorlinks=true, linkcolor=blue, filecolor=magenta, urlcolor=cyan}
\begin{document}

\title{On The Frames In Hilbert $C^{\ast}$-modules}

\author[M. Ghiati]{M'hamed Ghiati$^{1}$}
\address{$^{1}$ Laboratory of Analysis, Geometry and
	Applications (LAGA), Departement of Mathematics,  Ibn Tofail University, B.P.  133,  Kenitra,  Morocco}
\curraddr{}
\email{mhamed.ghiati@uit.ac.ma;\\ mohammed.mouniane@uit.ac.ma}
\thanks{}
\author[M. Mouniane]{Mohammed Mouniane$^{*1}$}

\author[M. Rossafi]{Mohamed Rossafi$^{2}$}
\address{$^{2}$ LASMA Laboratory,  Departement of Mathematics,  Faculty of Sciences Dhar  El Mahraz,  Sidi Mohamed Ben Abdellah University,  Fes,  Morocco}
\curraddr{}
\email{rossafimohamed@gmail.com; mohamed.rossafi@usmba.ac.ma}
\thanks{}

 \subjclass[2010]{Primary: 42C15; Secondary: 47A05}

 \keywords{Frame, $\ast$-frame, g-frame, $\ast$-g-frame, $\ast$-$K$-g-frame, operator frame, Hilbert $C^{\ast}$-modules.}

\begin{abstract} 
	Frame theory has been rapidly generalized and various generalizations have been developed. In this paper, we present a brief survey of the frames in Hilbert $C^{\ast}$-modules, including frames, $\ast$-frames, g-frames, $\ast$-g-frames, $\ast$-$K$-g-frame, operator frame and $\ast$-$K$-operator in Hilbert $C^{\ast}$-modules. Various proofs are given for some results. Also, we prove some new results. Moreover, non-trivial examples are presented.

\end{abstract}

\maketitle

\section{Introduction} 

One of the important concepts in the study of vector spaces is the concept of a basis for the vector space, which allows every vector to be uniquely represented as a linear combination of the basis elements. However, the linear independence property for a basis is restrictive; sometimes it is impossible to find vectors that both fulfill the basis requirements and also satisfy external conditions demanded by applied problems. For such purposes, we need to look for more flexible types of spanning sets. Frames provide these alternatives. They not only have a great variety for use in applications but also have a rich theory from a pure analysis point of view.

A frame is a set of vectors in a Hilbert space that can be used to reconstruct each vector in the space from its inner products with the frame vectors. These inner products are generally called the frame coefficients of the vector. But unlike an orthonormal basis, each vector may have infinitely many different representations in terms of its frame coefficients.
Frames for Hilbert spaces were introduced by Duffin and Schaefer \cite{Duf} in 1952 to study some deep problems in nonharmonic Fourier series by abstracting the fundamental notion of Gabor \cite{Gab} for signal processing. In fact, in $1946$ Gabor, showed that any function $f\in L^{2}(\mathbb{R})$ can be reconstructed via a Gabor system $\{g(x-ka)e^{2\pi imbx}: k,m\in\mathbb{Z}\}$ where $g$ is a continuous compact support function. These ideas did not generate much interest outside of nonharmonic Fourier series and signal processing until the landmark paper of Daubechies, Grossmann, and Meyer \cite{Gross} in 1986, where they developed the class of tight frames for signal reconstruction and they showed that frames can be used to find series expansions of functions in $L^{2}(\mathbb{R})$ which are very similar to the expansions using orthonormal bases. After this innovative work, the theory of frames began to be widely studied. While orthonormal bases have been widely used for many applications, it is the redundancy that makes frames useful in applications.
Formally, a frame in a separable Hilbert space $\mathcal{H}$ is a sequence $\{f_{i}\}_{i\in I}$  for which there exist positive constants $A, B > 0$ called frame bounds such that
\begin{equation*}  
A\lVert x\lVert^{2}\leq\sum_{i\in I}|\langle x,f_{i}\rangle|^{2}\leq B\lVert x\lVert^{2};\forall x\in\mathcal{H}.
\end{equation*}
It is remarkable that the above inequalities imply the existence of a dual frame $\{\tilde{f}_{i}\}_{i\in I}$ , such that the following reconstruction formula holds for every $x\in\mathcal{H}$: $x=\sum_{i\in I}\langle x,\tilde{f}_{i}\rangle f_{i}$.
In particular, any orthonormal basis for $\mathcal{H}$ is a frame. However, in general, a frame need not be a basis and most useful frames are over-complete. The redundancy that frames carry is what makes them very useful in many applications.

Hilbert space frames have been traditionally used in signal processing because of their resilience to additive noise, resilience to quantization, the numerical stability of reconstruction, and their ability to capture important signal characteristics. 
Today, frame theory is an exciting, dynamic, and fast-paced subject with applications to a wide variety of areas in mathematics and engineering, including sampling theory, operator theory, harmonic analysis, nonlinear sparse approximation, pseudodifferential operators, wavelet theory, wireless communication, data transmission with erasures, filter banks, signal processing, image processing, geophysics, quantum computing, sensor networks, and more.
The last decades have seen tremendous activity in the development of frame theory and many generalizations of frames have come into existence.
 Hilbert $C^{\ast}$-modules is a generalization of Hilbert spaces by allowing the inner product to take values in a $C^{\ast}$-algebra rather than in the field of complex numbers.
\section{Preliminaries} 
\subsection{$\mathcal{C}^{\ast}$-algebra}
\begin{definition}\cite{Con}.
	If $\mathcal{A}$ is a Banach algebra, an involution is a map $ a\rightarrow a^{\ast} $ of $\mathcal{A}$ into itself such that for all $a$ and $b$ in $\mathcal{A}$ and all scalars $\alpha$ the following conditions hold:
	\begin{enumerate}
		\item  $(a^{\ast})^{\ast}=a$.
		\item  $(ab)^{\ast}=b^{\ast}a^{\ast}$.
		\item  $(\alpha a+b)^{\ast}=\bar{\alpha}a^{\ast}+b^{\ast}$.
	\end{enumerate}
\end{definition}
\begin{definition}\cite{Con}.
	A $\mathcal{C}^{\ast}$-algebra $\mathcal{A}$ is a Banach algebra with involution such that $$\|a^{\ast}a\|=\|a\|^{2}$$ for every $a$ in $\mathcal{A}$.
\end{definition}
\begin{Exs}
\begin{enumerate}
		\item	$ B(\mathcal{H}) $	the algebra of bounded operators on a Hilbert space $\mathcal{H}$, is a  $\mathcal{C}^{\ast}$-algebra, where for each operator $A$, $A^{\ast}$ is the adjoint of $A$.
	\item $C(X)$ the algebra of continuous functions on a compact space $X$, is an abelian $\mathcal{C}^{\ast}$-algebra, where $f^{\ast}(x):=\overline{f(x)}$.
	\item $C_{0}(X)$ the algebra of continuous functions on a locally compact space $X$ that vanish at infinity is an abelian $\mathcal{C}^{\ast}$-algebra, where $f^{\ast}(x):=\overline{f(x)}$.
\end{enumerate}	
\end{Exs}
\begin{definition}\cite{Con}.
	An element $a$ in a $\mathcal{C}^{\ast}$-algebra $\mathcal{A}$ is positive if $a^{\ast}=a$ and $sp(a)\subset\mathbb{R}^{+}$. We write $a\geq 0$ if $a$ is positive. The set of all positive
	elements in $\mathcal{A}$ will be denoted by $\mathcal{A}^{+}$.
	\begin{lemma}\cite{Xu2008}
	Let $ T \in End^{*}_{\mathcal{A}}(\mathcal{H},\mathcal{K}) $ be  a bounded operator with closed range  $\mathcal{R} (T)$.  Then there exists a bounded operator $ U^{\dagger} \in End^{*}_{\mathcal{A}}( \mathcal{K}, \mathcal{H}) $ for which 
	$$ TT^{\dagger}x=x      ,  \qquad x\in\mathcal{R} (T). $$
\end{lemma}	
\end{definition}
\begin{lemma} \label{1} \cite{Pas}.
	
	Let $\mathcal{H}$ be a Hilbert $\mathcal{A}$-module. If $T\in End_{\mathcal{A}}^{\ast}(\mathcal{H})$, then $$\langle Tx,Tx\rangle\leq\|T\|^{2}\langle x,x\rangle, \forall x\in\mathcal{H}.$$
\end{lemma}

\begin{lemma} \label{sb} \cite{Ara}
	Let $\mathcal{H}$ and $\mathcal{K}$ be  two Hilbert $\mathcal{A}$-modules and $T\in End_{\mathcal{A}}^{\ast}(\mathcal{H},\mathcal{K})$. Then the following statements are equivalent:
	\begin{itemize}
		\item [(i)] $T$ is surjective.
		\item [(ii)] $T^{\ast}$ is bounded below with respect to norm, i.e., there exists $m>0$ such that $\|T^{\ast}x\|\geq m\|x\|$ for all $x\in\mathcal{K}$.
		\item [(iii)] $T^{\ast}$ is bounded below with respect to the inner product, i.e., there exists $m'>0$ such that $\langle T^{\ast}x,T^{\ast}x\rangle\geq m'\langle x,x\rangle$ for all $x\in\mathcal{K}$.
	\end{itemize}
\end{lemma}

\begin{lemma} \label{3} \cite{Ali}
	Let $\mathcal{H}$ and $\mathcal{K}$ be two Hilbert $\mathcal{A}$-modules and $T\in End_{\mathcal{A}}^{\ast}(\mathcal{H},\mathcal{K})$. 
	\begin{itemize}
		\item [(i)] If $T$ is injective and $T$ has closed range, then the adjointable map $T^{\ast}T$ is invertible and $$\|(T^{\ast}T)^{-1}\|^{-1}\leq T^{\ast}T\leq\|T\|^{2}.$$
		\item  [(ii)]	If $T$ is surjective, then the adjointable map $TT^{\ast}$ is invertible and $$\|(TT^{\ast})^{-1}\|^{-1}\leq TT^{\ast}\leq\|T\|^{2}.$$
	\end{itemize}	
\end{lemma}

\subsection{Hilbert $\mathcal{C}^{\ast}$-modules}
\begin{definition}\cite{Kap}.
	Let $ \mathcal{A} $ be a unital $C^{\ast}$-algebra and $\mathcal{H}$ be a left $ \mathcal{A} $-module, such that the linear structures of $\mathcal{A}$ and $ \mathcal{H} $ are compatible. $\mathcal{H}$ is a pre-Hilbert $\mathcal{A}$-module if $\mathcal{H}$ is equipped with an $\mathcal{A}$-valued inner product $\langle.,.\rangle :\mathcal{H}\times\mathcal{H}\rightarrow\mathcal{A}$, such that is sesquilinear, positive definite and respects the module action. In the other words,
	\begin{enumerate}
		\item $ \langle x,x\rangle\geq0 $ for all $ x\in\mathcal{H} $ and $ \langle x,x\rangle=0$ if and only if $x=0$.
		\item $\langle ax+y,z\rangle=a\langle x,z\rangle+\langle y,z\rangle$ for all $a\in\mathcal{A}$ and $x,y,z\in\mathcal{H}$.
		\item $ \langle x,y\rangle=\langle y,x\rangle^{\ast} $ for all $x,y\in\mathcal{H}$.
	\end{enumerate}	 
\end{definition}
For $x\in\mathcal{H}, $ we define $||x||=||\langle x,x\rangle||^{\frac{1}{2}}$. If $\mathcal{H}$ is complete with $||.||$, it is called a Hilbert $\mathcal{A}$-module or a Hilbert $C^{\ast}$-module over $\mathcal{A}$. For every $a$ in $C^{\ast}$-algebra $\mathcal{A}$, we have $|a|=(a^{\ast}a)^{\frac{1}{2}}$ and the $\mathcal{A}$-valued norm on $\mathcal{H}$ is defined by $|x|=\langle x, x\rangle^{\frac{1}{2}}$ for $x\in\mathcal{H}$.
\begin{Exs}
	\begin{enumerate}
	\item Let $\mathcal{H}$ be a Hilbert space, then $B(\mathcal{H})$ is a Hilbert $C^{\ast}$-module with the inner product $\langle T, S\rangle=TS^{\ast}, \forall T, S\in B(\mathcal{H})$.
	\item Let $\mathcal{H}$ and $\mathcal{K}$ be separable Hilbert spaces and let $B(\mathcal{H,K})$ be the set of all
	bounded linear operators from $\mathcal{H}$ into $\mathcal{K}$. Then $B(\mathcal{H,K})$ is a Hilbert $B(\mathcal{K})$-module  with the inner product $\langle T, S\rangle=TS^{\ast}, \forall T, S\in B(\mathcal{H,K})$.
	\item Let $X$ be a locally compact Hausdorff space and $\mathcal{H}$ a Hilbert space, the Banach space $C_{0}(X, \mathcal{H})$ of all continuous $\mathcal{H}$-valued functions vanishing at infinity is a Hilbert $C^{\ast}$-module over the $C^{\ast}$-algebra $C_{0}(X)$ with inner product $\langle f, g\rangle(x):= \langle f(x), g(x)\rangle$ and module operation $(\phi f)(x)= \phi(x)f(x)$, for all $\phi\in C_{0}(X)$ and $f\in C_{0}(X, \mathcal{H})$.
	\item If $ \{\mathcal{H}_{k}\}_{k\in\mathbb{N}} $ is a countable set of Hilbert $\mathcal{A}$-modules, then one can define their direct sum $ \oplus_{k\in\mathbb{N}}\mathcal{H}_{k} $. On the $\mathcal{A}$-module $ \oplus_{k\in\mathbb{N}}\mathcal{H}_{k} $ of all sequences $x=(x_{k})_{k\in\mathbb{N}}: x_{k}\in\mathcal{H}_{k}$, such that the series $ \sum_{k\in\mathbb{N}}\langle x_{k}, x_{k}\rangle_{\mathcal{A}} $ is norm-convergent in the $\mathcal{C}^{\ast}$-algebra $\mathcal{A}$, we define the inner product by
	\begin{equation*}
	\langle x, y\rangle:=\sum_{k\in\mathbb{N}}\langle x_{k}, y_{k}\rangle_{\mathcal{A}} 
	\end{equation*}
	for $x, y\in\oplus_{k\in\mathbb{N}}\mathcal{H}_{k} $.\\
	Then $\oplus_{k\in\mathbb{N}}\mathcal{H}_{k}$ is a Hilbert $\mathcal{A}$-module.
	\\
	The direct sum of a countable number of copies of a Hilbert $\mathcal{C}^{\ast}$-module $\mathcal{H}$ is denoted by $l^{2}(\mathcal{H})$.
	\end{enumerate}
\end{Exs}
Let $\mathcal{H}$ and $\mathcal{K}$ be two Hilbert $\mathcal{A}$-modules, A map $T:\mathcal{H}\rightarrow\mathcal{K}$ is said to be adjointable if there exists a map $T^{\ast}:\mathcal{K}\rightarrow\mathcal{H}$ such that $\langle Tx,y\rangle_{\mathcal{A}}=\langle x,T^{\ast}y\rangle_{\mathcal{A}}$ for all $x\in\mathcal{H}$ and $y\in\mathcal{K}$.
\\
We also reserve the notation $End_{\mathcal{A}}^{\ast}(\mathcal{H},\mathcal{K})$ for the set of all adjointable operators from $\mathcal{H}$ to $\mathcal{K}$ and $End_{\mathcal{A}}^{\ast}(\mathcal{H},\mathcal{H})$ is abbreviated to $End_{\mathcal{A}}^{\ast}(\mathcal{H})$.

Throughout the paper, we consider a unital $C^{\ast}$-algebra.
\section{Frames In Hilbert $\mathcal{A}$-modules}
\subsection{Frames}
\begin{definition} 
	\cite{F4}. Let $ \mathcal{H} $ be a Hilbert $\mathcal{A}$-module. A family $\{x_{i}\}_{i\in I}$ of elements of $\mathcal{H}$ is a frame for $ \mathcal{H} $, if there exist two positive constants $A$, $B$, such that for all $x\in\mathcal{H}$,
	\begin{equation}\label{1}
	A\langle x,x\rangle_{\mathcal{A}}\leq\sum_{i\in I}\langle x,x_{i}\rangle_{\mathcal{A}}\langle x_{i},x\rangle_{\mathcal{A}}\leq B\langle x,x\rangle_{\mathcal{A}}.
	\end{equation}
	The numbers $A$ and $B$ are called lower and upper bounds of the frame, respectively. If $A=B=\lambda$, the frame is $\lambda$-tight. If $A = B = 1$, it is called a normalized tight frame or a Parseval frame. If the sum in the middle of \eqref{1} is convergent in norm, the frame is called standard.
\end{definition}
\begin{Ex}
	For $a\in\mathbb{R}$, define the translation operator
	\begin{equation*}
	T_{a}:L^{2}(\mathbb{R})\to L^{2}(\mathbb{R}),\;T_{a}f(x)=f(x-a)
	\end{equation*}
	For $b\in\mathbb{R}$, define the modulation operator
	\begin{equation*}
	E_{b}:L^{2}(\mathbb{R})\to L^{2}(\mathbb{R}),\;E_{b}f(x)=e^{2\pi ibx}f(x)
	\end{equation*}
	A frame for $L^{2}(\mathbb{R})$ of the form
	\begin{equation*}
	\{E_{mb}T_{na}g\}_{m,n\in\mathbb{Z}}=\{e^{2\pi imbx}g(x-na)\}_{m,n\in\mathbb{Z}}
	\end{equation*}
	is called a Gabor frame.
\end{Ex}
Let $\{x_{i}\}_{i\in I}$ be a frame of a finitely or countably generated Hilbert $\mathcal{A}$-module $\mathcal{H}$ over a unital $\mathcal{C}^{\ast}$-algebra $\mathcal{A}$.
\\
The operator $T: \mathcal{H} \to l^{2}(\mathcal{A})$ defined by $Tx=\{\langle x, x_{i}\rangle\}_{i\in I}$, is called the analysis operator.
\\
The adjoint operator $T^{\ast}: l^{2}(\mathcal{A}) \to \mathcal{H}$ is given by $T^{\ast}\{c_{i}\}_{i\in I}= \sum_{i\in I}c_{i}x_{i}$. $T^{\ast}$ is called pre-frame operator or the synthesis operator.
\\
By composing $T$ and $T^{\ast}$ we obtain the frame operator $S: \mathcal{H} \to \mathcal{H}$ which is given by
\begin{equation*}
 Sx=T^{\ast}Tx=\sum_{i\in I}\langle x, x_{i}\rangle x_{i}.
\end{equation*}
\begin{Prop}
	The frame operator $S$ is positive, selfadjoint, bounded and invertible.
\end{Prop}
\begin{proof} For all $x\in \mathcal{H}$ we have,
	\begin{equation*}
		\langle Sx,x\rangle_{\mathcal{A}} = \langle T^{\ast}Tx,x\rangle_{\mathcal{A}}=\langle \sum_{i\in I}\langle x, x_{i}\rangle x_{i},x\rangle_{\mathcal{A}}=\sum_{i\in I}\langle x, x_{i}\rangle\langle x_{i},x\rangle_{\mathcal{A}} .
	\end{equation*}
	 we have,
	 $$0\leq	A\langle x,x\rangle_{\mathcal{A}}\leq\langle Sx;x\rangle\leq B\langle x,x\rangle_{\mathcal{A}}.$$
	 \\
	Then $S$ is a positive operator, also, it's sefladjoint.\\
	 We have,
	\begin{equation*}
		A\langle x,x\rangle_{\mathcal{A}} \leq\langle Sx,x\rangle_{\mathcal{A}}\leq B\langle x,x\rangle_{\mathcal{A}}, \quad  x\in \mathcal{H}.
	\end{equation*}
	So,
	\begin{equation*}
		A.I_{\mathcal{H}} \leq S_{C}\leq  B.I_{\mathcal{H}}
	\end{equation*}
	Then $S_{C}$ is a bounded operator.\\
	Moreover, 
	\begin{equation*}
		0 \leq I_{\mathcal{H}}-B^{-1}S_{C} \leq \frac{B-A}{B}.I_{\mathcal{H}},
	\end{equation*} 
	Consequently,
	\begin{equation*}
		\|I_{\mathcal{H}}-B^{-1}S_{C} \|=\underset{x \in \mathcal{H}, \|x\|=1}{\sup}\|\langle(I_{\mathcal{H}}-B^{-1}S_{C})x,x\rangle_{\mathcal{A}} \|\leq \frac{B-A}{B}<1.
	\end{equation*}
	The  shows that  $S$ is invertible.
\end{proof}
The frame operator is positive, invertible, and is the unique operator in $End_{\mathcal{A}}^{\ast}(\mathcal{H})$ such that the reconstruction formula
\begin{equation*}
x=\sum_{i\in I}\langle x, S^{-1}x_{i}\rangle x_{i}=\sum_{i\in I}\langle x, x_{i}\rangle S^{-1}x_{i},
\end{equation*}
holds for all $x\in\mathcal{H}$.
\\
The sequences $\{S^{-1}x_{i}\}_{i\in I}$ and $\{S^{-\frac{1}{2}}x_{i}\}_{i\in I}$ are frames for $\mathcal{H}$.
\\
The frame $\{S^{-1}x_{i}\}_{i\in I}$ is said to be the canonical dual frame of $\{x_{i}\}_{i\in I}$ and the frame $\{S^{-\frac{1}{2}}x_{i}\}_{i\in I}$ is said to be the canonical Parseval frame of $\{x_{i}\}_{i\in I}$.
\\
The following theorem give a characterization of standard frame.
\begin{theorem}
	Let $\mathcal{H}$ be a finitely or countably generated Hilbert $\mathcal{A}$-module over a unital $\mathcal{C}^{\ast}$-algebra $\mathcal{A}$, and $\{x_{i}\}_{i}\subset\mathcal{H}$ a sequence such that $\sum_{i\in I}\langle x,x_{i}\rangle_{\mathcal{A}}\langle x_{i},x\rangle_{\mathcal{A}}$ converges in norm for every $x\in\mathcal{H}$. Then $\{x_{i}\}_{i}$ is a frame of $\mathcal{H}$ with bounds $A$ and $B$ if and only if
	\begin{equation*}
	A\|x\|^{2}\leq\Bigg\|\sum_{i\in I}\langle x,x_{i}\rangle_{\mathcal{A}}\langle x_{i},x\rangle_{\mathcal{A}}\Bigg\|\leq B\|x\|^{2}
	\end{equation*}
	for all $x\in\mathcal{H}$.
\end{theorem}
\begin{proof}
	Suppose that $ \{x_{i}:i\in I \}$ is a frame in Hilbert $ \mathcal{A}$-module with bounds $A$ and $B$ $\mathcal{H}$  We have
	\begin{equation*}
		A\langle x,x\rangle_{\mathcal{A}}\leq\sum_{i\in I}\langle x,x_{i}\rangle_{\mathcal{A}}\langle x_{i},x\rangle_{\mathcal{A}}\leq B\langle x,x\rangle_{\mathcal{A}}.
	\end{equation*}
 Since $\langle x,x\rangle \geq 0$  we have $$ A\lVert\langle  x ,x \rangle_{\mathcal{A}}\lVert\leq \lVert \sum_{i\in I}\langle x,x_{i}\rangle_{\mathcal{A}}\langle x_{i},x\rangle_{\mathcal{A}} \lVert\leq B\lVert\langle x, x \rangle_{\mathcal{A}}\lVert .$$
	We have $$ A\lVert x\lVert^{2} \leq \lVert \sum_{i\in I}\langle x,x_{i}\rangle_{\mathcal{A}}\langle x_{i},x\rangle_{\mathcal{A}} \lVert\leq B\lVert x \lVert^{2} .$$
	Now suppose that \eqref{*}  holds, we know that the  frame operator $S$ is positive  self-adjoint and inversible $$ \langle S^{1/2}x,S^{1/2}x\rangle =\langle Sx,x\rangle =\langle\sum_{i\in I} \langle x ,x_{i}  \rangle_{\mathcal{A}}x_{i},x\rangle =\sum_{i\in I}\langle x,x_{i}\rangle_{\mathcal{A}}\langle x_{i},x\rangle_{\mathcal{A}}.$$
	Hence $$\sqrt{A}\lVert x\lVert \leq \lVert S^{1/2}x\lVert \leq \sqrt{B}\lVert x \lVert $$
	$$ A_1\langle  x , x \rangle_{\mathcal{A}}\leq \sum_{i\in I} \langle x ,x_{i}  \rangle_{\mathcal{A}}\langle x ,x_{i}  \rangle_{\mathcal{A}}\leq B_1 \langle x, x \rangle_{\mathcal{A}}. $$ Which implies that
	$\{x_{i}:i\in I \}$ is a frame in Hilbert $ \mathcal{A}$-module $\mathcal{H}$.
\end{proof}
\begin{theorem} \label{2.3}
	Let $\{x_{i} : i\in I\}$ be a frame for $\mathcal{H}$, with lower and upper bounds $A$ and $B$, respectively. Then the frame transform $T:H\rightarrow l^{2} \left(\left\{\mathcal{V}_{i}\right\}\right)$ defined by: $Tx=\{\langle x,x_{i}\rangle: i\in I\}$ is injective and  adjointable, and  has a closed range  with $\|T\|\leq\|B\|^{\frac{1}{2}}$. The adjoint operator $T^{\ast}$ is surjective, given by $T^{\ast}x=\sum_{i\in I}c_{i}x_{i}$, where $x=\{x_{i}\}_{i\in I}$.
\end{theorem}

\begin{proof}
	Let $x\in H$. By the definition of frame for $\mathcal{H}$, we have
	\begin{equation}\label{1}
		A\langle x,x\rangle_{\mathcal{A}}\leq\sum_{i\in I}\langle x,x_{i}\rangle_{\mathcal{A}}\langle  x_{i},x\rangle_{\mathcal{A}}\leq B\langle x,x\rangle_{\mathcal{A}}.
	\end{equation}
	We have
	\begin{equation*}
	A\langle x,x\rangle_{\mathcal{A}}\leq\langle  \sum_{i\in I}\langle x,x_{i}\rangle_{\mathcal{A}}x_{i},x\rangle_{\mathcal{A}}\leq B\langle x,x\rangle_{\mathcal{A}}.
	\end{equation*}
	We have
	\begin{equation*}
		A\langle x,x\rangle \leq\langle T^{*}Tx,x\rangle \leq B\langle x,x\rangle.
	\end{equation*}
	So
	\begin{equation}\label{1}
		A\langle x,x\rangle \leq\langle Tx,Tx\rangle\leq B\langle x,x\rangle .
	\end{equation}
	If $Tx=0$ then $\langle x,x\rangle=0$ and so $x=0$, i.e., $T$ is injective.
	
	We now show that the range of $T$ is closed. Let $\{Tx_{n}\}_{n\in\mathbb{N}}$ be a sequence in the range of $T$ such that $\lim_{n\rightarrow\infty}Tx_{n}=y.$\\
	By \eqref{1}, we have, for $n, m\in\mathbb{N}$,
	\begin{equation*}
		\|A\langle x_{n}-x_{m},x_{n}-x_{m}\rangle \|\leq\|\langle T(x_{n}-x_{m}),T(x_{n}-x_{m})\rangle\|=\|T(x_{n}-x_{m})\|^{2}.
	\end{equation*}
	Since $\{Tx_{n}\}_{n\in\mathbb{N}}$ is Cauchy sequence in $H$,
	$\|A\langle x_{n}-x_{m},x_{n}-x_{m}\rangle \|\rightarrow0$, as $n,m\rightarrow\infty.$\\
	Note that for $n, m\in\mathbb{N}$,
	\begin{align*}
		\|\langle x_{n}-x_{m},x_{n}-x_{m}\rangle\|&= \|A^{-1}A\langle x_{n}-x_{m},x_{n}-x_{m}\rangle \|\\ 
		&\leq  \|A^{-1}\|\|A\langle x_{n}-x_{m},x_{n}-x_{m}\rangle \|.
	\end{align*}
	Therefore the sequence $\{x_{n}\}_{n\in\mathbb{N}}$ is Cauchy and hence there exists $x\in H$ such that $x_{n}\rightarrow x$ as $n\rightarrow\infty$. Again by \eqref{1} we have $\|T(x_{n}-x)\|^{2}\leq\|B\|\|\langle x_{n}-x,x_{n}-x\rangle\|$.\\
	Thus $\|Tx_{n}-Tx\|\rightarrow0$ as $n\rightarrow\infty$ implies that $Tx=y$. It concludes that the range of $T$ is closed.\\
	For all $x\in H$, $y=\{y_{i}\}\in l^{2}\left(\left\{\mathcal{V}_{i}\right\}\right)$, we have
	\begin{equation*}
		\langle Tx,y\rangle=\langle \{\langle x, x_{i}\rangle\}_{i\in I}x,y_{i}\rangle =\left\langle x,\sum_{i\in I}\langle x, x_{i}\rangle y_{i}\right\rangle.
	\end{equation*}
	Then $T$ is adjointable and $T^{\ast}y=\sum_{i\in I}\langle x, x_{i}\rangle y_{i}.$
	By \eqref{1} we have $\|Tx\|^{2}\leq\|B\|\|x\|^{2}$ and so $\|T\|\leq\|B\|^{\frac{1}{2}}.$
	By \eqref{1} we have $\|Tx\|\geq\|A^{-1}\|^{-1}\|x\|$ for all  $ x\in \mathcal{H}$ and so by Lemma \ref{5}, $T^{\ast}$ is surjective.\\
	This completes the proof.
\end{proof}
\begin{theorem}
	Let $\{x_{i}\}_{i\in I}$ be a frame for $\mathcal{H}$ with frame transform $T$. Then $\{x_{i}\}_{i\in I}$ is a frame for $\mathcal{H}$ with lower and upper frame bounds $||(T^{\ast}T)^{-1}||^{-1}$ and $||T||^{2}$, respectively.
\end{theorem}
\begin{proof}
	By Theorem \ref{2.3}, $T$ is injective and has a closed range, and so  by Lemma \ref{3}, 
	\begin{equation*}
		||(T^{\ast}T)^{-1}||^{-1}\langle x,x\rangle\leq \langle T^{\ast}Tx,x\rangle\leq ||T||^{2}\langle x,x\rangle,\qquad\forall x\in U.    
	\end{equation*}
	So 
	\begin{equation*}
		||(T^{\ast}T)^{-1}||^{-1}\langle x,x\rangle\leq \langle \sum_{i\in I}\langle x,x_{i}\rangle_{\mathcal{A}} x_{i},x\rangle_{\mathcal{A}}\leq ||T||^{2}\langle x,x\rangle,\qquad\forall x\in U.     
	\end{equation*}
	Hence $\{x_{i}\}_{i\in I}$ is a frame for $\mathcal{H}$ with lower and upper frame bounds $||(T^{\ast}T)^{-1}||^{-1}$ and $||T||^{2}$, respectively.
\end{proof}
\section{g-Frames}
\begin{definition}
	\cite{AB}. We call a sequence $\{\Lambda_{i}\in End_{\mathcal{A}}^{\ast}(\mathcal{H},V_{i}):i\in I \}$ a g-frame in Hilbert $\mathcal{A}$-module $\mathcal{H}$ with respect to $\{V_{i}:i\in I \}$ if there exist two positive constants $C$, $D$, such that for all $x\in\mathcal{H}$, 
	\begin{equation}\label{2}
	C\langle x,x\rangle_{\mathcal{A}}\leq\sum_{i\in I}\langle \Lambda_{i}x,\Lambda_{i}x\rangle_{\mathcal{A}}\leq D\langle x,x\rangle_{\mathcal{A}}.
	\end{equation}
	The numbers $C$ and $D$ are called lower and upper bounds of the g-frame, respectively. If $C=D=\lambda$, the g-frame is $\lambda$-tight. If $C = D = 1$, it is called a g-Parseval frame. If the sum in the middle of \eqref{2} is convergent in norm, the g-frame is called standard.
\end{definition}
\begin{Ex}
	Let $\mathbb{C}^{2}$ be the Hilbert $\mathbb{C}^{2}$-module  with $\mathbb{C}^{2}$-inner product $$\langle(x_{1},x_{2}),(y_{1},y_{2})\rangle=(x_{1}\bar{y_{1}},x_{2}\bar{y_{2}})$$ and let $\mathcal{A}$ be the totality of all diagonal operators $diag\{a,b\}$ on $\mathbb{C}^{2}$, sending $(z_{1},z_{2})^{t}$ to $(az_{1},bz_{2})^{t}$. Fix $\{a_{i}\}_{i}$ and $\{b_{i}\}_{i}$ in $l^{2}$. Define $\Lambda_{i}:\mathbb{C}^{2}\to\mathbb{C}^{2}, (z_{1},z_{2})\to(a_{i}z_{1},b_{i}z_{2})$.\\
	Then $\{\Lambda_{i}\}_{i}$ is a g-frame for $\mathbb{C}^{2}$ with bounds
\begin{center}
		$\min\{\sum_{i}|a_{i}|^{2},\sum_{i}|b_{i}|^{2}\}$ and $\max\{\sum_{i}|a_{i}|^{2},\sum_{i}|b_{i}|^{2}\}$,
\end{center} respectively.
\end{Ex}
Like frames, we define the frame transform $T$, the synthesis operator $T^{\ast}$ and the g-frame operator $S$ as follws: $T: \mathcal{H} \to \oplus_{i\in I}V_{i}$, $Tx=\{\Lambda_{i}x\}_{i\in I}$,
$T^{\ast}: \oplus_{i\in I}V_{i} \to \mathcal{H}$, $T^{\ast}y=\sum_{i\in I}\Lambda_{i}^{\ast}y_{i}$ for all $y=\{y_{i}\}_{i\in I}$ in $\oplus_{i\in I}V_{i}$,
and $S=T^{\ast}T: \mathcal{H} \to \mathcal{H}$ is given by $Sx=\sum_{i\in I}\Lambda_{i}^{\ast}\Lambda_{i}x$ for each $x\in\mathcal{H}$. 
\\
The $g$-frame operator is positive, invertible, and the follow reconstruction formula holds
\begin{equation*}
x=\sum_{i\in I}\Lambda_{i}^{\ast}\Lambda_{i}S^{-1}x=\sum_{i\in I}S^{-1}\Lambda_{i}^{\ast}\Lambda_{i}x,
\end{equation*}
holds for all $x\in\mathcal{H}$.

The following theorem give a characterization of standard $g$-frame.
\begin{theorem}\label{*}
	Let $\Lambda_{i}\in End_{\mathcal{A}}^{\ast}(\mathcal{H},V_{i})$ for any $i\in I$ and $\sum_{i\in I}\langle \Lambda_{i}x,\Lambda_{i}x\rangle_{\mathcal{A}}$ converge in norm for $x\in\mathcal{H}$. Then $\{\Lambda_{i}\}_{i\in I}$ is a g-frame for $\mathcal{H}$ with respect to $\{V_{i}\}_{i\in I}$ if and only if there exist constants $A$, $B>0$ such that
		\begin{equation*}
	A\|x\|^{2}\leq\Bigg\|\sum_{i\in I}\langle \Lambda_{i}x,\Lambda_{i}x\rangle_{\mathcal{A}}\Bigg\|\leq B\|x\|^{2}
	\end{equation*}
	for all $x\in\mathcal{H}$.
\end{theorem}
\begin{proof}
	Suppose that $ \{\Lambda_{i}:i\in I \}$ is a $g$-frame in Hilbert $ \mathcal{A}$-module $\mathcal{H}$ with respect to $\{V_{i}\}_{i\in I}$. We have
	$$ A\langle  x , x \rangle_{\mathcal{A}}\leq \sum_{i\in I} \langle \Lambda_{i}x ,\Lambda_{i}x  \rangle_{\mathcal{A}}\leq B\langle x, x \rangle_{\mathcal{A}}, $$ Since $\langle x,x\rangle \geq 0$,  we have $$ A\lVert\langle  x ,x \rangle_{\mathcal{A}}\lVert\leq \lVert \sum_{i\in I} \langle \Lambda_{i}x ,\Lambda_{i}x  \rangle_{\mathcal{A}} \lVert\leq B\lVert\langle f, f \rangle_{\mathcal{A}}\lVert .$$
	We have $$ A\lVert x\lVert^{2} \leq \lVert \sum_{i\in I} \langle \Lambda_{i}x ,\Lambda_{i}x  \rangle_{\mathcal{A}} \lVert\leq B\lVert x \lVert^{2} .$$
	Now suppose that \eqref{*}  holds, we know that the $g-$ frame operator $S_{\Lambda}$ is positive  self-adjoint and inversible $$ \langle S^{1/2}x,S^{1/2}x\rangle =\langle Sx,x\rangle =\sum_{i\in I} \langle \Lambda_{i}x ,\Lambda_{i}x  \rangle_{\mathcal{A}},$$
	hence $$\sqrt{A}\lVert x\lVert \leq \lVert S^{1/2}x\lVert \leq \sqrt{B}\lVert x \lVert $$
	$$ A_1\langle  x , x \rangle_{\mathcal{A}}\leq \sum_{i\in I} \langle \Lambda_{i}x ,\Lambda_{i}x  \rangle_{\mathcal{A}}\leq B_1 \langle x, x \rangle_{\mathcal{A}}. $$ Which implies that
	$\{\Lambda_{i}:i\in I \}$ is a $g$-frame in Hilbert $ \mathcal{A}$-module $\mathcal{H}$ with respect to $\{V_{i}\}_{i\in I}$.
\end{proof}
\begin{theorem}
	If  $ \Lambda:=\{  \Lambda_{i}\in{End^{*}_{\mathcal{A}}(\mathcal{H}, \mathcal{H}_{i})} \}_{{i}\in {I}}$ is a $g$-frame in Hilbert $\mathcal{A}$-module  $\mathcal{H}$  with respect to $  \{\mathcal{H}_{i}\}_{{i}\in {I}}$ with bounds $A$ and $B$,then the  $g$-frame operator $ S\colon \mathcal{H}\rightarrow \mathcal{H}$ defined by
	$$ Sf= \displaystyle\sum_{i\in I} \Lambda^{*}_{i}\Lambda_{i}x $$
	is a bounded invertible operator.\\
	Let $\widetilde\Lambda_{i}=\Lambda_{i}S^{-1} $. Then  $ \{\widetilde  \Lambda_{i}\in{End^{*}_{\mathcal{A}}(\mathcal{H}, \mathcal{H}_{i})} \}_{{i}\in {I}}$
	is also a $g$-frame in Hilbert $\mathcal{A}$-module  $\mathcal{H}$, with respect to $  \{\mathcal{H}_{i}\}_{{i}\in {I}}$  with bounds $B^{-1}$ and $A^{-1}$ and satisfies
	$$  x= \displaystyle\sum_{i\in I} \Lambda^{*}_{i}\widetilde\Lambda_{i}x =  \displaystyle\sum_{i\in I} \widetilde\Lambda^{*}_{i}\Lambda_{i}x, \qquad x\in\mathcal{H}.$$ 
\end{theorem}
\begin{proof}
For $x\in\mathcal{H}$.	We have
	$$
	\begin{aligned}
		\left\|\displaystyle\sum_{i \in I_{1}} \Lambda_{i}^{*}\Lambda_{i} x\right\| &=\sup _{y \in \mathcal{H},\|y\|=1}\left\lVert\left\langle\displaystyle\sum_{i \in I_{1}}\Lambda_{i}^{*} \Lambda_{i} x, y\right\rangle_{\mathcal{A}}\right\lVert\\
		&=\sup _{y \in \mathcal{H},\|y\|=1}\left\lVert\displaystyle\sum_{i \in I_{1}}\left\langle \Lambda_{i} x ,\Lambda_{i} y \right\rangle_{\mathcal{A}}\right\lVert \\
		& \leq \sup _{x \in \mathcal{H},\|y\|=1} \left\|\displaystyle\sum_{i \in I_{1}} \langle \Lambda_{i}x ,\Lambda_{i}x \rangle_{\mathcal{A}}\right\|^{1 / 2}\left\|\displaystyle\sum_{i \in I_{1}}\langle \Lambda_{i} x ,\Lambda_{i} y\rangle_{\mathcal{A}}\right\|^{1/2} \\
		& \leq \sqrt{B}\left\lVert\displaystyle\sum_{i \in I_{1}}\langle \Lambda_{i} y ,\Lambda_{i} y\rangle_{\mathcal{A}}\right\|^{1/2}.
	\end{aligned}$$
	Hence, the series in ($\sum_{i \in I_{1}}\Lambda_{i}^{*}\Lambda_{i} $) are convergent. Therefore, $Sx$ is well defined For any $x \in \mathcal{H}$.\\
	On the other hand, it is easy to check that for any $x,g \in \mathcal{H}$\\
	$$ \langle Sx,g\rangle_{\mathcal{A}} =\displaystyle\sum_{i \in I_{1}}\langle \Lambda_{i}^{*} \Lambda_{i} x, g \rangle_{\mathcal{A}} =\displaystyle\sum_{i \in I_{1}}\langle x, \Lambda_{i}^{*} \Lambda_{i} g \rangle_{\mathcal{A}} =\langle f,Sg\rangle_{\mathcal{A}} .$$
	Hence, $S$ is a self-adjoint operator.\\
	And therefore,
	 \begin{align*}
	 	\lVert S\lVert &= \sup _{x \in \mathcal{H},\|x\|=1} \lVert\langle Sx,x \rangle_{\mathcal{A}}\lVert\\ &=\sup _{x \in \mathcal{H},\|x\|=1} \lVert\langle \displaystyle\sum_{i\in I} \Lambda^{*}_{i}\Lambda_{i}x,x \rangle_{\mathcal{A}}\lVert \\
	 	&=\sup _{x \in \mathcal{H},\|x\|=1} \lVert \displaystyle\sum_{i\in I}\langle  \Lambda_{i}x ,\Lambda_{i}x \rangle_{\mathcal{A}}\lVert \leq B .
	 \end{align*}
	So, $ S$ is a bounded operator.\\
	Since there is $\langle x,x\rangle_{\mathcal{A}} \geq 0$, then for all $ x \in \mathcal{H}$,
	$$ A\lVert \langle x,x\rangle_{\mathcal{A}} \lVert \leq\lVert \langle Sx,x\rangle_{\mathcal{A}} \lVert \leq B\lVert \langle x,x\rangle_{\mathcal{A}} \lVert ,$$
	this implies that $$ A\lVert x \lVert ^{2}\leq\lVert \langle Sx,x\rangle_{\mathcal{A}} \lVert \leq B\lVert x \lVert^{2} .$$
	Then $$ A\lVert x\lVert^{2} \leq\lVert \langle Sx,x\rangle_{\mathcal{A}} \lVert \leq \lVert Sx \lVert\lVert x\lVert.$$
	It follows that  $$ \lVert Sx\lVert \geq A\lVert x\lVert.$$
	Which implies that $S$ is injective and $S\mathcal{H}$ is closed in $\mathcal{H}.$ 
	Let $ g\in \mathcal{H}$ be such that $\langle Sx,g\rangle_{\mathcal{A}} =0$ for any $x\in \mathcal{H}$.\\ 
	Then, we have $$\langle x,Sg\rangle_{\mathcal{A}} =0,\qquad x \in \mathcal{H}.$$
	This implies that $ Sg=0$ and therefore $g=0$. Hence $S\mathcal{H}= \mathcal{H}$. \\
	Consequently, $S$ is invertible  and $$\lVert S^{-1} \lVert \leq \dfrac{1}{A}.$$
	For any $x\in\mathcal{H},$ we have $$ x= SS^{-1}x= S^{-1}Sx= \displaystyle\sum_{i\in I} \Lambda^{*}_{i}\Lambda_{i}S^{-1}x=\displaystyle\sum_{i\in I}S^{-1} \Lambda^{*}_{i}\Lambda_{i}x.$$
	Let $\widetilde\Lambda_{i}=\Lambda_{i}S^{-1} $. 
	Then the above equalities become
	$$  x= \displaystyle\sum_{i\in I} \Lambda^{*}_{i}\widetilde\Lambda_{i}x =  \displaystyle\sum_{i\in I} \widetilde\Lambda^{*}_{i}\Lambda_{i}x.$$
	We prove now  that $ \{\widetilde  \Lambda_{i}\in{End^{*}_{\mathcal{A}}(\mathcal{H}, \mathcal{H}_{i})} \}_{{i}\in {I}}$
	is also a $g$-frame in Hilbert $\mathcal{A}$-module  $\mathcal{H}$ with respect to $  \{\mathcal{H}_{i}\}_{{i}\in {I}}$. 
	In fact, for any $x \in \mathcal{H}$, we have
	\begin{align*}
		\left\| \displaystyle\sum_{i \in I}\widetilde{\Lambda_{i}} x\right\| ^{2} 
		&=\left\| \displaystyle\sum_{i \in I}\left\langle \Lambda_{i}S^{-1}x,  \Lambda_{i}S^{-1}x\right\rangle_{\mathcal{A}}\right\|  \\
		&=\left\| \displaystyle\sum_{i \in I} \langle\Lambda_{i}^{*} \Lambda_{i}S^{-1} x,S^{-1} x\rangle_{\mathcal{A}}\right\|   \\
		&=\left\| \langle S^{-1}Sx,S^{-1}x \rangle_{\mathcal{A}}\right\| \\
		&=\left\|  \langle \Lambda_{i}^{*} \Lambda_{i}S^{-1} x,S^{-1} x\rangle_{\mathcal{A}}\right\|  \\
		&=\lVert\langle S^{-1}Sx,S^{-1}x \rangle_{\mathcal{A}}\lVert \\
		&= \left\|  \langle f,S^{-1}x\rangle_{\mathcal{A}}\right\|  
		\leq \dfrac{1}{A}\lVert\langle x,x\rangle_{\mathcal{A}}\lVert.
	\end{align*}
	On the other hand, since
	$$
	\begin{aligned}
		\lVert x\lVert^{2} &=\lVert\langle x,x \rangle_{\mathcal{A}}\lVert\\
		&= \lVert\langle \displaystyle\sum_{i \in I} \widetilde{\Lambda_{i}}^{*}\Lambda_{i}x, x\rangle_{\mathcal{A}}\lVert\\ 
		&=\lVert\displaystyle\sum_{i \in I}\langle \Lambda_{i}x,\widetilde{\Lambda_{i}}x  \rangle_{\mathcal{A}}\lVert \\
		&\leq\lVert \displaystyle\sum_{i \in I}\langle  \Lambda_{i}x,\Lambda_{i}x \rangle_{\mathcal{A}}\lVert^{1 / 2}\lVert \displaystyle\sum_{i \in I}\langle  \widetilde{\Lambda_{i}}x,\widetilde{\Lambda_{i}}x \rangle_{\mathcal{A}}\lVert^{1 / 2} \\
		&\leq \sqrt{B}\lVert x\lVert \lVert \displaystyle\sum_{i \in I}\langle  \widetilde{\Lambda_{i}}x,\widetilde{\Lambda_{i}}x \rangle_{\mathcal{A}}\lVert^{1 / 2}. 
	\end{aligned}
	$$
	We have $$\lVert \displaystyle\sum_{i \in I}\langle  \widetilde{\Lambda_{i}}x,\widetilde{\Lambda_{i}}x \rangle_{\mathcal{A}}\lVert \geq \dfrac{1}{B}\lVert\langle x,x\rangle_{\mathcal{A}}\lVert.$$ 
	Hence, $ \{\widetilde  \Lambda_{i}\in{End^{*}_{\mathcal{A}}(\mathcal{H}, \mathcal{H}_{i})} \}_{{i}\in {I}}$
	is a $g$-frame in Hilbert $\mathcal{A}$-module  $\mathcal{H}$ with frame bounds $\dfrac{1}{A}$ and $\dfrac{1}{B}$.
	Let $\widetilde{S}$ be the $g-$frame operator associated with $ \Lambda:=\{  \Lambda_{i}\in \}_{{i}\in {I}}$. Then, we have
	$$
	\begin{aligned}
		S\widetilde{S}x&=  \displaystyle\sum_{i\in I} S \widetilde{\Lambda^{*}_{i}}\widetilde{\Lambda_{i}}x=\displaystyle\sum_{i\in I}SS^{-1} \Lambda^{*}_{i}\Lambda_{i}S^{-1}x\\
		&=\displaystyle\sum_{i\in I}\Lambda^{*}_{i}\Lambda_{i}S^{-1}x =SS^{-1} \Lambda^{*}_{i}x= x.
	\end{aligned}
	$$
	Hence $\widetilde{S}=S^{-1}$ and  $\widetilde{\Lambda_{i}}\widetilde{S^{-1}}S=\Lambda_{i} .$
\end{proof}
\begin{theorem}
	Let $\{\Lambda_{i}\}_{i\in I}$ and $\{\Gamma_{i}\}_{i\in I}$ be $g$-Bessel sequences for Hilbert $C^{\ast}$-modules $U_{1}$ and $U_{2}$ with $g$-Bessel bounds $B_{1}$ and $B_{2}$, respectively. Then $\{\Lambda_{w}^{\ast}\Gamma_{w}\}_{w\in\Omega}$ is a-$g$-Bessel sequence for $U_{2}$ with respect to $U_{1}$.
\end{theorem}

\begin{proof}
	We have for each $x\in U_{2}$,
	\begin{align*}
		\sum_{i \in I}\langle \Lambda_{i}^{\ast}\Gamma_{i}x,\Lambda_{i}^{\ast}\Gamma_{i}x\rangle_{\mathcal{A}} &\leq \sum_{i \in I}||\Lambda_{i}^{\ast}||^{2}\langle \Gamma_{i}x,\Gamma_{i}x\rangle_{\mathcal{A}} \\&\leq ||B_{1}||^{2}\sum_{i \in I}\langle \Gamma_{i}x,\Gamma_{i}x\rangle_{\mathcal{A}} \\&\leq ||B_{1}||^{2}B_{2}\langle x,x\rangle_{\mathcal{A}} \\&\leq  ||B_{1}||B_{2}\langle x,x\rangle_{\mathcal{A}} .
	\end{align*}
	Hence $\{\Lambda_{i}^{\ast}\Gamma_{i}\}_{i\in I}$ is a $g$-Bessel sequence for $U_{2}$ with respect to $U_{1}$.
\end{proof}
\begin{theorem}
	\label{Lemma 2.5} A sequence $\{  \Lambda_{i}\in{End^{*}_{\mathcal{A}}(\mathcal{H}, \mathcal{V}_{i})} \}_{{i}\in {I}}$ is a g-frame  in Hilbert $\mathcal{A}$  module  $\mathcal{H}$ with respect to $\left\{V_{i}\right\}_{i \in I}$ if and only if
	$$
	Q:\left\{g_{i}\right\}_{i \in I} \rightarrow \displaystyle\sum_{i \in I} \Lambda_{i}^{*} g_{}
	$$
	is a well defined bounded linear operator from $l^{2}\left(\left\{\mathcal{H}_{i}\right\}_{i \in I}\right)$ onto $\mathcal{H}$, where the ${g}$ -frame bounds are $\left\|Q^{+}\right\|^{-2},\|Q\|^{2}$ and $Q^{+}$ is the pseudo-inverse of $Q$.
\end{theorem}
\textit{Proof.} (1) \(\Rightarrow\) If $ \Lambda:=\{  \Lambda_{i}\in{End^{*}_{\mathcal{A}}(\mathcal{H}, \mathcal{H}_{i})} \}_{{i}\in {I}}$ is a $g$-frame  in Hilbert $\mathcal{A}$  module  $\mathcal{H}$ with respect to \(\left\{\mathcal{H}_{i}\right\}_{i \in I}\) with bounds \(A\) and \(B\), then for any finite subset \(I_{1} \subset I\), we have
\begin{align*}
	\left\|\displaystyle\sum_{i \in I_{1}} \Lambda_{i}^{*} g_{j}\right\| &=\sup _{f \in \mathcal{H},\|f\|=1}\left\lVert\left\langle\displaystyle\sum_{i \in I_{1}} \Lambda_{i}^{*} g_{i}, f\right\rangle\right\lVert\\
	&=\sup _{f \in \mathcal{H},\|f\|=1}\left\lVert\displaystyle\sum_{i \in I_{1}}\left\langle g_{i}, \Lambda_{i} f\right\rangle\right\lVert \\
	& \leq \sup _{f \in \mathcal{H},\|f\|=1} \lVert \displaystyle\sum_{i \in I_{1}}\langle g_{i},g_{i}\rangle\lVert^{1 / 2}\lVert \displaystyle\sum_{i \in I_{1}}\langle\Lambda_{i} f , \Lambda_{i} f \rangle\lVert^{1 / 2} \\
	& \leq \sqrt{B}\lVert \displaystyle\sum_{i \in I_{1}}\langle g_{i},g_{i}\rangle\lVert^{1 / 2}.
\end{align*}
Hence the series \(\displaystyle\sum_{i \in i} \Lambda_{i}^{*} g_{i}\) converges in \(\mathcal{H} .\) Therefore the operator defined by (8) is well defined from \(l^{2}\left(\left\{\mathcal{V}_{i}\right\}_{i \in I}\right)\) into \(\mathcal{H}\) with \(\|Q\| \leq \sqrt{B}\).

For every \(f \in \mathcal{H}\), from Lemma 2.2, there exists a \(g \in \mathcal{H}\) such that \(f=S g=\displaystyle\sum_{i \in I} \Lambda_{i}^{*} \Lambda_{i} g\). Since $\{  \Lambda_{i}\in{End^{*}_{\mathcal{A}}(\mathcal{H}, \mathcal{H}_{i})} \}_{{i}\in {I}}$ is a g-frame for \(\mathcal{H}\) with respect to \(\left\{\mathcal{H}_{i}\right\}_{i \in I}\), then \(\left\{\Lambda_{i} g\right\}_{i \in I} \in\)
\(l^{2}\left(\left\{\mathcal{V}_{i}\right\}_{i \in I}\right)\) and \(Q\left(\left\{\Lambda_{i} g\right\}_{i \in I}\right)=\displaystyle\sum_{i \in I} \Lambda_{i}^{*} \Lambda_{i} g=f .\) This implies that the operator \(Q\) is onto.

\((2) \Leftarrow\) If \(Q\) is a well defined bounded linear operator from \(l^{2}\left(\left\{\mathcal{H}_{i}\right\}_{i \in I}\right)\) onto \(\mathcal{H}\), then for any \(f \in \mathcal{H}\) and any finite subset \(I_{1} \subset I\), we have
\begin{align*}
	\left\|\displaystyle\sum_{i \in I_{1}}\Lambda_{i} f\right\|^{2} &=\left\|\displaystyle\sum_{i \in I_{1}}\left\langle\Lambda_{i} f, \Lambda_{i} f\right\rangle\right\|\\
	&=\left\|\displaystyle\sum_{i \in I_{1}}\left\langle f, \Lambda_{i}^{*} \Lambda_{i} f\right\rangle \right\|  \\
	&=\left \lVert\langle f, \displaystyle\sum_{i \in I_{1}} \Lambda_{i}^{*} \Lambda_{i} f\rangle\right\|\\
	&\leq \lVert\langle f,f\rangle\lVert^{1 / 2}\lVert\langle\displaystyle\sum_{i \in I_{1}} \Lambda_{i}^{*} \Lambda_{i} f,\displaystyle\sum_{i \in I_{1}} \Lambda_{i}^{*} \Lambda_{i} f \rangle\lVert^{1 / 2}\\
	& \leq\|f\|\left\|\displaystyle\sum_{i \in I_{1}} \Lambda_{i}^{*} \Lambda_{i} f\right\|=\|f\|\left\|Q\left(\left\{\Lambda_{i} f\right\}_{i \in I_{1}}\right)\right\| \\
\end{align*}
It follows that
\(\left\|\displaystyle\sum_{i \in I_{1}}\Lambda_{i} f\right\|^{2} \leq\|Q\|^{2}\|f\|^{2}\) for all \(f \in \mathcal{H}\) and any finite subset \(I_{1} \subset I .\) Hence we obtain
$$
\left\|\displaystyle\sum_{i \in I}\Lambda_{i} f\right\|^{2} \leq\|Q\|^{2}\|f\|^{2}, \quad \forall f \in \mathcal{H}
$$
On the other hand, since \(Q\left(\ell^{2}\left(\left\{\mathcal{H}_{i}\right\}_{i \in I}\right)\right)=\mathcal{H}\), from Lemma 2.1, there exists a unique bounded operator \(Q^{+}: \mathcal{H} \rightarrow l^{2}\left(\left\{\mathcal{H}_{i}\right\}_{i \in I}\right)\) satisfying \(Q Q^{+} f=f, f \in Q\left(\l^{2}\left(\left\{\mathcal{V}_{i}\right\}_{i \in I}\right)\right)=\mathcal{H} .\) Let
\(Q^{+} f=\left\{a_{i}\right\}_{i \in I}\). Then we have
$$
\begin{array}{l}
	\left\|\displaystyle\sum_{i \in I}a_{i}\right\|^{2}=\left\|Q^{+} f\right\|^{2} \leq\left\|Q^{+}\right\|^{2}\|f\|^{2}, \quad f \in \mathcal{H}, \\
	f=Q Q^{+} f=\displaystyle\sum_{i \in I} \Lambda_{i}^{*} a_{i}, \quad f \in \mathcal{H}.
\end{array}
$$
Hence we obtain
$$
\begin{aligned}
	\|f\|^{4} &=\lVert\langle f, f\rangle\lVert^{2}=\left\lVert\left\langle\displaystyle\sum_{i \in I} \Lambda_{i}^{*} a_{i}, f\right\rangle\right\lVert^{2}=\left\lVert\displaystyle\sum_{i \in I}\left\langle a_{i}, \Lambda_{i} f\right\rangle\right\lVert^{2} \\
	& \leq\left\|Q^{+}\right\|^{2}\|f\|^{2}\left\| \displaystyle\sum_{i \in I}\Lambda_{i} f\right\|^{2}.
\end{aligned}
$$
This implies that
$$
\frac{1}{\left\|Q^{+}\right\|^{2}}\|f\|^{2} \leq\left\| \displaystyle\sum_{i \in I}\Lambda_{i} f\right\|^{2}, \quad \forall f \in \mathcal{H}.
$$
\begin{theorem}
	\label{theorem 3.3} Let a sequence  $ \Lambda:=\{  \Lambda_{i}\in{End^{*}_{\mathcal{A}}(\mathcal{H}, \mathcal{V}_{i})} \}_{{i}\in {I}}$ be a g-frame for $\mathcal{H}$ with respect to $\left\{\mathcal{V}_{i}\right\}_{i \in I}$ with bounds $A$ and $B$. If $P$ is the orthogonal projection from $l^{2}\left(\left\{\mathcal{V}_{i}\right\}_{i \in I}\right)$ onto $R_{Q^{*}}$,
	then $P$ is defined by $P\left(\left\{g_{i}\right\}_{i \in I}\right):=\left\{\Lambda_{i} \displaystyle\sum_{k \in I} \widetilde{\Lambda}_{k}^{*} g_{k}\right\}_{i \in I}$ for any
	$\left\{g_{i}\right\}_{i \in I}$ belongs to $l^{2}\left(\left\{\mathcal{V}_{i}\right\}_{i \in I}\right)$.
\end{theorem}
\begin{proof}
	Let $\tilde{T}$ be the operator defined by
	$$
	l^{2}\left(\left\{\mathcal{H}_{i}\right\}_{i \in I}\right) \rightarrow R_{Q^{*}}, \widetilde{T}\left(\left\{g_{i}\right\}_{i \in I}\right):=\left\{\Lambda_{i} \displaystyle\sum_{k \in I} \widetilde{\Lambda}_{k}^{*} g_{k}\right\}_{i \in I}.
	$$
	Firstly, we prove that $\tilde{T}$ is a bounded linear operator. For all $\left\{g_{i}\right\}_{i \in I},\left\{h_{i}\right\}_{i \in I} \in l^{2}\left(\left\{\mathcal{H}_{i}\right\}_{i \in I}\right)$, we consider that
	$$
	\begin{aligned}
		\tilde{T}\left(\left\{g_{i}\right\}_{i \in I}+\left\{h_{i}\right\}_{i \in I}\right) 
		&= \widetilde{T}\left(\left\{g_{i}+h_{i}\right\}_{i \in I}\right)\\
		&=\left\{\Lambda_{i} \displaystyle\sum_{k \in I} \widetilde{\Lambda}_{k}^{*}\left(g_{k}+h_{k}\right)\right\}_{i \in I} \\
		&=\left\{\Lambda_{i}\left(\displaystyle\sum_{k \in I} \widetilde{\Lambda}_{k}^{*} g_{k}+\displaystyle\sum_{k \in I} \tilde{\Lambda}_{k}^{*} h_{k}\right)\right\}_{i \in I}\\
		&=\left\{\Lambda_{i} \displaystyle\sum_{k \in I} \widetilde{\Lambda}_{k}^{*} g_{k}\right\}_{i \in I}+\left\{\Lambda_{i} \displaystyle\sum_{k \in I} \widetilde{\Lambda}_{k}^{*} h_{k}\right\}_{i \in I} \\
		& =\tilde{T}\left(\left\{g_{i}\right\}_{i \in I}\right)+\tilde{T}\left(\left\{h_{i}\right\}_{i \in I}\right)
	\end{aligned}
	$$
	and 
	$$
	\begin{aligned}
		\left\|\tilde{T}\left(\left\{g_{i}\right\}_{i \in I}\right)\right\|^{2}
		=&\left\|\left\{\Lambda_{i} \displaystyle\sum_{k \in I} \widetilde{\Lambda}_{k}^{*} g_{k}\right\}_{i \in I}\right\| ^{2}=\left\|\left\{\Lambda_{i} \displaystyle\sum_{k \in I} S^{-1} \Lambda_{k}^{*} g_{k}\right\}_{i \in I} \right\|^{2} \\
		=& \displaystyle\sum_{i \in I}\left\|\Lambda_{i} \displaystyle\sum_{k \in I} S^{-1} \Lambda_{k}^{*} g_{k}\right\|^{2} \leq B\left\|\displaystyle\sum_{k \in I} S^{-1} \Lambda_{k}^{*} g_{k}\right\|^{2} \\
		\leq & \frac{B}{A}\left\|\left\{g_{k}\right\}_{k \in I}\right\|^{2} .
	\end{aligned}
	$$
	Hence, $\tilde{T}$ is a bounded linear operator. Secondly, let $g=\displaystyle\sum_{i \in I} \widetilde{\Lambda}_{i}^{*} g_{i}$, since
	$$
	\begin{aligned}
		\tilde{T}^{2}\left(\left\{g_{i}\right\}_{i \in I}\right)&= \widetilde{T}\left(\left\{\Lambda_{i} g\right\}_{i \in I}\right)=\left\{\Lambda_{I} \displaystyle\sum_{k \in I} \tilde{\Lambda}_{k}^{*} \Lambda_{k} g\right\}_{i \in I} \\
		&=\left\{\Lambda_{i} \displaystyle\sum_{k \in I} S^{-1} \Lambda_{k}^{*} \Lambda_{k} g\right\}_{i \in I}=\left\{\Lambda_{i} S^{-1} \displaystyle\sum_{k \in I} \Lambda_{k}^{*} \Lambda_{k} g\right\}_{i \in I} \\
		&=\left\{\Lambda_{i} g\right\}_{i \in I}=\widetilde{T}\left(\left\{g_{i}\right\}_{i \in I}\right),
	\end{aligned}
	$$
	we obtain $\tilde{T}^{2}=\tilde{T}$.\\ Hence, we have that $\tilde{T}$ is a projection from $l^{2}\left(\left\{\mathcal{V}_{i}\right\}_{i \in I}\right)$ onto $R_{Q^{2}}$.
	Finally, for all $\left\{g_{i}\right\}_{i \in I},\left\{f_{i}\right\}_{i \in I} \in l^{2}\left(\left\{\mathcal{V}_{i}\right\}_{i \in I}\right)$, we obtain
	$$
	\begin{aligned}
		\left\langle\tilde{T}\left(\left\{g_{i}\right\}_{i \in I}\right),\left\{f_{i}\right\}_{i \in I}\right\rangle_{\mathcal{A}} &=\left\langle\left\{\Lambda_{i} \displaystyle\sum_{k \in I} \tilde{\Lambda}_{k}^{*} g_{k}\right\}_{i \in I},\left\{f_{i}\right\}_{i \in I}\right\rangle_{\mathcal{A}} \\
		&=\displaystyle\sum_{i \in I}\left\langle\Lambda_{i} \displaystyle\sum_{k \in I} \widetilde{\Lambda}_{k}^{*} g_{k}, f_{i}\right\rangle_{\mathcal{A}} \\
		&=\displaystyle\sum_{i \in I}\left\langle\displaystyle\sum_{k \in I} \widetilde{\Lambda}_{k}^{*} g_{k}, \Lambda_{i}^{*} f_{i}\right\rangle_{\mathcal{A}}\\
		&=\left\langle\displaystyle\sum_{k \in I} \widetilde{\Lambda}_{k}^{*} g_{k}, \displaystyle\sum_{i \in I} \Lambda_{i}^{*} f_{i}\right\rangle_{\mathcal{A}}=\left\langle\displaystyle\sum_{k \in I} S^{-1} \Lambda_{k}^{*} g_{k}, \displaystyle\sum_{i \in I} \Lambda_{i}^{*} f_{i}\right\rangle_{\mathcal{A}} \\
		&=\left\langle S^{-1} \displaystyle\sum_{k \in I} \Lambda_{k}^{*} g_{k}, \displaystyle\sum_{i \in I} \Lambda_{i}^{*} f_{i}\right\rangle_{\mathcal{A}}=\left\langle\displaystyle\sum_{k \in I} \Lambda_{k}^{*} g_{k}, \displaystyle\sum_{i \in I} S^{-1} \Lambda_{i}^{*} f_{i}\right\rangle_{\mathcal{A}} \\
		&=\left\langle\displaystyle\sum_{k \in I} \Lambda_{k}^{*} g_{k}, \displaystyle\sum_{i \in I} \widetilde{\Lambda}_{i}^{*} f_{i}\right\rangle_{\mathcal{A}}=\displaystyle\sum_{k \in I}\left\langle g_{k}, \Lambda_{k} \displaystyle\sum_{i \in I} \widetilde{\Lambda}_{i}^{*} f_{i}\right\rangle_{\mathcal{A}} \\
		&=\left\langle\left\{g_{k}\right\}_{k \in I},\left\{\Lambda_{k} \displaystyle\sum_{i \in I} \widetilde{\Lambda}_{i}^{*} f_{i}\right\}_{k \in I}\right\rangle_{\mathcal{A}} \\
		&=\left\langle\left\{g_{i}\right\}_{i \in I}, \tilde{T}\left(\left\{f_{i}\right\}_{i \in I}\right)\right\rangle_{\mathcal{A}} .
	\end{aligned}
	$$
	It shows that $\widetilde{T}^{*}=\widetilde{T}$. Hence, $\tilde{T}$ is an orthogonal projection from $l^{2}\left(\left\{\mathcal{V}_{i}\right\}_{i \in I}\right)$ onto $R_{Q^{*}}$. Since orthogonal projection is unique, we
	have $P=\widetilde{T}$.
\end{proof}
\section{$\ast$-Frames}
\begin{definition}
	\cite{Ali}. Let $ \mathcal{H} $ be a Hilbert $\mathcal{A}$-module over a unital $C^{\ast}$-algebra. A family $\{x_{i}\}_{i\in I}$ of elements of $\mathcal{H}$ is an $\ast$-frame for $ \mathcal{H} $, if there	exist strictly nonzero elements $A$, $B$ in $\mathcal{A}$, such that for all $x\in\mathcal{H}$,
	\begin{equation}\label{3}
	A\langle x,x\rangle_{\mathcal{A}} A^{\ast}\leq\sum_{i\in I}\langle x,x_{i}\rangle_{\mathcal{A}}\langle x_{i},x\rangle_{\mathcal{A}}\leq B\langle x,x\rangle_{\mathcal{A}} B^{\ast}.
	\end{equation}
	The elements $A$ and $B$ are called lower and upper bounds of the $\ast$-frame, respectively. If $A=B=\lambda_{1}$, the $\ast$-frame is $\lambda_{1}$-tight. If $A = B = 1$, it is called a normalized tight $\ast$-frame or a Parseval $\ast$-frame. If the sum in the middle of \eqref{3} is convergent in norm, the $\ast$-frame is called standard.
\end{definition}
\begin{Ex}
	Let $\mathcal{A}$ be the $\mathcal{C}^{\ast}$-algebra of the set of all diagonal matrices in $M_{2,2}(\mathbb{C})$ and suppose $\mathcal{A}$ is the Hilbert $\mathcal{A}$-module over itself. Consider $A_{i}=\begin{bmatrix}
	\frac{1}{2^{i}}&0\\0&\frac{1}{3^{i}}
	\end{bmatrix}$ for all $i\in\mathbb{N}$. For $A=\begin{bmatrix}
	a&0\\0&b
	\end{bmatrix}\in\mathcal{A}$, we have
	\begin{equation*}
	\sum_{i\in\mathbb{N}}\langle A,A_{i}\rangle\langle A_{i},A\rangle=\begin{bmatrix}
	\frac{|a|^{2}}{3}&0\\0&\frac{|b|^{2}}{8}
	\end{bmatrix}=\begin{bmatrix}
	\frac{1}{\sqrt{3}}&0\\0&\frac{1}{\sqrt{8}}
	\end{bmatrix}\langle A,A\rangle\begin{bmatrix}
	\frac{1}{\sqrt{3}}&0\\0&\frac{1}{\sqrt{8}}
	\end{bmatrix}
	\end{equation*} Then $\{A_{i}\}_{i\in\mathbb{N}}$ is $\begin{bmatrix}
	\frac{1}{\sqrt{3}}&0\\0&\frac{1}{\sqrt{8}}
	\end{bmatrix}$-tight $\ast$-frame for Hilbert $\mathcal{A}$-module $\mathcal{A}$.
\end{Ex}
\begin{Rem}
	\begin{enumerate}
		\item The set of all frames in Hilbert $\mathcal{A}$-modules can be considered as a subset of $\ast$-frames. 
		\item We see that $\ast$-frames can be studied as frames with different bounds.
	\end{enumerate}
\end{Rem}
Now we define the $\ast$-frame operator and compare its properties with ordinary case.
\begin{definition}
	Let $\{x_{i}\}_{i\in I}$ be an $\ast$-frame for $\mathcal{H}$ with pre-$\ast$-frame operator $T$ and lower and upper $\ast$-frame bounds $A$ and $B$, respectively. The $\ast$-frame operator $S: \mathcal{H} \to \mathcal{H}$ is defined by $Sx=T^{\ast}Tx=\sum_{i\in I}\langle x, x_{i}\rangle_{\mathcal{A}}x_{i}$.
\end{definition}
The $\ast$-frame operator has some similar properties with frame operator in ordinary frames, but the other properties are different. The main cause of differences is $\mathcal{A}$-valued bounds. However, the reconstruction formula is given from the $\ast$-frame operator.
\begin{theorem}
Let $\{x_{i}\}_{i\in I}$ be an $\ast$-frame for $\mathcal{H}$ with $\ast$-frame operator $S$ and lower and upper $\ast$-frame bounds $A$ and $B$, respectively. Then $S$ is positive, invertible and adjointable. Also, the following inequality $\|A^{-1}\|^{-2}\leq\|S\|\leq\|B\|^{2}$ holds, and the reconstruction formula $x=\sum_{i\in I}\langle x, S^{-1}x_{i}\rangle_{\mathcal{A}}x_{i}$ holds $\forall x\in\mathcal{H}$.
\end{theorem}
In the following corollary we see that $\ast$-frames can be studied as frames with different bounds.
\begin{corollary}
Let $\{x_{i}\}_{i\in I}$ be an $\ast$-frame for $\mathcal{H}$ with pre-$\ast$-frame operator $T$ and lower and upper $\ast$-frame bounds $A$ and $B$, respectively. Then $\{x_{i}\}_{i\in I}$ is a frame for $\mathcal{H}$ with lower and upper frame bounds $\|(T^{\ast}T)^{-1}\|^{-1}$ and $\|T\|^{2}$, respectively.
\end{corollary}
\subsection{$\ast$-g-Frames}

The following definition was introduced independently by Alijani \cite{A} and Bounader \cite{Kab}, wich is a generalization of g-frames.
\begin{definition}
	\cite{A, Kab}. We call a sequence $\{\Lambda_{i}\in End_{\mathcal{A}}^{\ast}(\mathcal{H},V_{i}):i\in I \}$ an $\ast$-g-frame in Hilbert $\mathcal{A}$-module $\mathcal{H}$ over a unital $C^{\ast}$-algebra with respect to $\{V_{i}:i\in I \}$ if there exist strictly nonzero elements $A$, $B$ in $\mathcal{A}$, such that for all $x\in\mathcal{H}$, 
	\begin{equation}\label{4}
	A\langle x,x\rangle_{\mathcal{A}} A^{\ast}\leq\sum_{i\in I}\langle \Lambda_{i}x,\Lambda_{i}x\rangle_{\mathcal{A}}\leq B\langle x,x\rangle_{\mathcal{A}} B^{\ast}.
	\end{equation}
	The elements $A$ and $B$ are called lower and upper bounds of the $\ast$-g-frame, respectively. If $A=B=\lambda_{1}$, the $\ast$-g-frame is $\lambda_{1}$-tight. If $A= B = 1$, it is called an $\ast$-g-Parseval frame. If the sum in the middle of \eqref{4} is convergent in norm, the $\ast$-g-frame is called standard.
\end{definition}
\begin{Ex}
	Let $\{x_{i}\}_{i\in\mathbb{I}}$ be an $\ast$-frame for $\mathcal{H}$ with bounds $A$ and $B$, respectively. For each $i\in\mathbb{I}$, we define $\Lambda_{i}:\mathcal{H}\to\mathcal{A}$ by $\Lambda_{i}x=\langle x,x_{i}\rangle_{\mathcal{A}},\;\; \forall x\in\mathcal{H}$. $\Lambda_{i}$ is adjointable and $\Lambda_{i}^{\ast}a=ax_{i}$ for each $a\in\mathcal{A}$. And we have 
	\begin{equation*}
	A\langle x,x\rangle_{\mathcal{A}} A^{\ast}\leq\sum_{i\in I}\langle x,x_{i}\rangle_{\mathcal{A}}\langle x_{i},x\rangle_{\mathcal{A}}\leq B\langle x,x\rangle_{\mathcal{A}} B^{\ast}, \forall x\in\mathcal{H}.
	\end{equation*}
	Then
	\begin{equation*}
	A\langle x,x\rangle_{\mathcal{A}} A^{\ast}\leq\sum_{i\in I}\langle T_{i}x,T_{i}x\rangle\leq B\langle x,x\rangle_{\mathcal{A}} B^{\ast}, \forall x\in\mathcal{H}.
	\end{equation*}
	So $\{\Lambda_{i}\}_{i\in\mathbb{I}}$ is an $\ast$-g-frame with bounds $A$ and $B$, respectively, in $\mathcal{H}$ with respect to $\mathcal{A}$.
\end{Ex}
\begin{Rem}
	\begin{enumerate}
		\item The set of all g-frames in Hilbert $\mathcal{A}$-modules can be considered as a subset of $\ast$-g-frames. 
		\item We see that $\ast$-g-frames can be studied as g-frames with different bounds.
	\end{enumerate}
	\end{Rem}
Now we define the $\ast$-g-frame operator.
\begin{definition}
	Let $\{\Lambda_{i}\in End_{\mathcal{A}}^{\ast}(\mathcal{H},V_{i}):i\in I \}$ be an $\ast$-g-frame for $\mathcal{H}$ with lower and upper $\ast$-g-frame bounds $A$ and $B$, respectively. The $\ast$-g-frame operator $S: \mathcal{H} \to \mathcal{H}$ is defined by $Sx=\sum_{i\in I}\Lambda_{i}^{\ast}\Lambda_{i}x$.
\end{definition}
The $\ast$-g-frame operator has some similar properties with g-frame operator, but the other properties are different. The main cause of differences is $\mathcal{A}$-valued bounds. However, the reconstruction formula is given from the $\ast$-g-frame operator.
\begin{theorem}
	Let $\{\Lambda_{i}\in End_{\mathcal{A}}^{\ast}(\mathcal{H},V_{i}):i\in I \}$ be an $\ast$-g-frame for $\mathcal{H}$ with $\ast$-g-frame operator $S$ and lower and upper $\ast$-g-frame bounds $A$ and $B$, respectively. Then $S$ is positive, invertible and adjointable. Also, the following inequality
	\begin{equation*}
	 \|A^{-1}\|^{-2}\leq\|S\|\leq\|B\|^{2}
	 \end{equation*}
	  holds, and the reconstruction formula 
	\begin{equation*}
	x=\sum_{i\in I}\Lambda_{i}^{\ast}\Lambda_{i}S^{-1}x=\sum_{i\in I}S^{-1}\Lambda_{i}^{\ast}\Lambda_{i}x,
	\end{equation*}
	holds for all $x\in\mathcal{H}$.
\end{theorem}
In the following corollary we see that $\ast$-g-frames can be studied as g-frames with different bounds.
\begin{corollary}
	Let $\{\Lambda_{i}\in End_{\mathcal{A}}^{\ast}(\mathcal{H},V_{i}):i\in I \}$ be an $\ast$-g-frame for $\mathcal{H}$ with pre-$\ast$-g-frame operator $T$ and lower and upper $\ast$-g-frame bounds $A$ and $B$, respectively. Then $\{\Lambda_{i}\in End_{\mathcal{A}}^{\ast}(\mathcal{H},V_{i}):i\in I \}$ is a g-frame for $\mathcal{H}$ with lower and upper g-frame bounds $\|(T^{\ast}T)^{-1}\|^{-1}$ and $\|T\|^{2}$, respectively.
\end{corollary}
\begin{theorem}\label{theorem}
	Let	$\{\Lambda_{i}\in End_{\mathcal{A}}^{\ast}(\mathcal{H},V_{i}): i\in I\}$.  If the operator $\theta:\oplus_{i\in I}V_{i}\rightarrow \mathcal{H}$,  defined by $\theta(\{x_{i}\}_{i\in I})=\sum_{i\in I}\Lambda_{i}^{\ast}x_{i}$,  is surjective, then  $\{\Lambda_{i}\}_{i\in I}$ is a  $\ast$-$g$-frame for $\mathcal{H}$.
\end{theorem}
\begin{proof}
	For each $x\in \mathcal{H}$, 
	\begin{align*}
		\left\|\sum_{i\in I}\langle \Lambda_{i}x,\Lambda_{i}x\rangle \right\|&=\left\|\sum_{i\in I}\langle x,\Lambda_{i}^{\ast}\Lambda_{i}x\rangle \right\|\\&=\left\|\langle x,\sum_{i\in I}\Lambda_{i}^{\ast}\Lambda_{i}x \rangle\right\|\\&\leq \|x\| \;\left\|\sum_{i\in I}\Lambda_{i}^{\ast}\Lambda_{i}x \right\|\\&\leq \|x\|\;\left\|\theta(\{\Lambda_{i}x\}_{i\in I})\right\|\\&\leq \|x\|\;\|\theta\|\;\left\|\{\Lambda_{i}x\}_{i\in I}\right\|\\&\leq\|x\|\;\|\theta\|\;\left\|\sum_{i\in I}\langle \Lambda_{i}x,\Lambda_{i}x \rangle \right\|^{\frac{1}{2}}.
	\end{align*}
	Thus  
	\begin{equation*}
		\left\|\sum_{i\in I}\langle \Lambda_{i}x,\Lambda_{i}x\rangle \right\|^{\frac{1}{2}}\leq\|\theta\|\;\|x\|.
	\end{equation*}
	So 
	\begin{equation}\label{eq8}
		\left\|\sum_{i\in I}\langle \Lambda_{i}x,\Lambda_{i}x\rangle \right\|\leq \|\theta\|^{2}\;\|x\|^{2},\qquad \forall x\in \mathcal{H}.
	\end{equation}
	Since $\theta$ is  surjective, by Lemma \ref{sb} there exists $\nu >0$ such that 
	\begin{equation*}
		||\theta^{\ast}x||\geq \nu ||x||,\qquad \forall x\in \mathcal{H}.
	\end{equation*}
	Therefore, $\theta^{\ast}$ is injective.  Hence $\theta^{\ast}: \mathcal{H}\rightarrow \mathcal{R}(\theta^{\ast})$ is invertible, and  for each $x\in \mathcal{H}$,
	$(\theta^{\ast}_{/\mathcal{R}(\theta^{\ast})})^{-1}\theta^{\ast}x=x$.
	
	So, for each $x\in \mathcal{H}$, 
	$$
	\|x\|=\|(\theta^{\ast}_{/\mathcal{R}(\theta^{\ast})})^{-1}\theta^{\ast}x\|\leq \|(\theta^{\ast}_{/\mathcal{R}(\theta^{\ast})})^{-1}\|\;\|\theta^{\ast}x\|.  
	$$
	Thus 
	\begin{equation}\label{eq12}
		\|(\theta^{\ast}_{/\mathcal{R}(\theta^{\ast})})^{-1}\|^{-2}\;\|x\|^{2}\leq \left\|\sum_{i\in I}\langle \Lambda_{i}x,\Lambda_{i}x\rangle \right\|.
	\end{equation}
	From \eqref{eq8} and \eqref{eq12}, $\{\Lambda_{i}\}_{i\in I}$ is a  $\ast$-$g$-frame for $\mathcal{H}$.
\end{proof}
\begin{theorem}
	Let $\{\Lambda_{i}\}_{i\in I}$ be a  $\ast$-$g$-frame for $\mathcal{H}$. If $\{\Gamma_{i}\}_{i\in I}$  is  a  $\ast$-$g$-Bessel sequence for $\mathcal{H}$ with respect to $\{V_{i}: {i\in I}\}$, and the operator $F:\mathcal{H}\rightarrow \mathcal{H}$, defined by $Fx=\sum_{i\in I}\Gamma_{i}^{\ast}\Lambda_{i}x$, is surjective, then $\{\Gamma_{i}\}_{i\in I}$ is a  $\ast$-$g$-frame for $\mathcal{H}$.
\end{theorem}
\begin{proof}
	Since $\{\Lambda_{i}\}_{i\in I}$ is  a  $\ast$-$g$-frame for $\mathcal{H}$, we have a  $\ast$-$g$-frame transform $T:\mathcal{H}\rightarrow \oplus_{i\in I}V_{i}$, defined by $Tx=\{\Lambda_{i}x\}_{i\in I}$.
	
	Now  the operator $K:\oplus_{i\in I}V_{w}\rightarrow \mathcal{H}$, defined  by $K(\{x_{i}\}_{i\in I})=\sum_{i\in I}\Gamma_{i}^{\ast}x_{i}$,  is well-defined,  since 
	\begin{align*}
		\left\|\sum_{i\in I}\Gamma_{i}^{\ast}x_{i}\right\|&=\sup_{\|y\|=1}\left\|\langle \sum_{i\in I}\Gamma_{i}^{\ast}x_{i}, y\rangle\right\|\\&=\sup_{\|y\|=1}\left\|\sum_{i\in I}\langle x_{i},\Gamma_{i}y\rangle \right\|\\&\leq\sup_{\|y\|=1}\left\|\sum_{i\in I}\langle x_{i},x_{i}\rangle  \right\|^{\frac{1}{2}}\left\|\sum_{i\in I}\langle \Gamma_{i}y,\Gamma_{i}y\rangle \right\|^{\frac{1}{2}}\\&\leq\sup_{\|y\|=1}\|\{x_{i}\}_{i\in I}\|\|\langle y,y\rangle \|^{\frac{1}{2}}= \|\{x_{i}\}_{i\in I}\|.
	\end{align*}
	We have for each $x\in \mathcal{H}$
	\begin{equation*}
		Fx=\sum_{i\in I}\Gamma_{i}^{\ast}\Lambda_{i}x=KTx.
	\end{equation*}
	Hence $F=KT$. Since $F$ is surjective,  for each $x\in \mathcal{H}$ there exists $y\in \mathcal{H}$ such that  $Fy=x$, which  implies $x=Fy=KTy$ and $Ty\in \oplus_{i\in I}V_{i}$ and so $K$ is surjective. From Theorem \ref{theorem},
	we conclude that $\{\Gamma_{i}\}_{i\in I}$ is a $\ast$-$g$-frame for $\mathcal{H}$.
\end{proof}

In the following we study continuous $\ast$-$g$-frames in two Hilbert $C^{\ast}$-modules with different $C^{\ast}$-algebras.

\begin{theorem}
	Let $(\mathcal{H},\mathcal{A}, \langle .,.\rangle_{\mathcal{A}})$ and $(\mathcal{H},\mathcal{B}, \langle .,.\rangle_{\mathcal{B}})$ be two Hilbert $C^{\ast}$-modules,  $\phi :\mathcal{A} \rightarrow \mathcal{B}$ be a $\ast$-homomorphism and $\theta$ be an adjointable map on $\mathcal{H}$ such that $\langle \theta x,\theta y\rangle_{\mathcal{B}}=\phi(\langle x,y\rangle_{\mathcal{A}})$ for all $x, y\in \mathcal{H}$. Also, suppose that $\{\Lambda_{i}\}_{i\in I}$ is a  $\ast$-$g$-frame for $(\mathcal{H}, \mathcal{A},\langle .,.\rangle_{\mathcal{A}})$ with  $\ast$-$g$-frame operator $S_{\mathcal{A}}$ and lower and upper bounds $A$, $B$ respectively. If $\theta$ is surjective and $\theta\Lambda_{i}=\Lambda_{i}\theta$ for all $i\in I$, then $\{\Lambda_{i}\}_{i\in I}$ is a  $\ast$-$g$-frame for $(\mathcal{H},\mathcal{B}, \langle .,.\rangle_{\mathcal{B}})$ with  $\ast$-$g$-frame operator $S_{\mathcal{B}}$ and lower and upper bounds $\phi(A)$ and $\phi(B)$, respectively, and $\langle S_{\mathcal{B}}\theta x,\theta y\rangle_{\mathcal{B}} =\phi(\langle S_{\mathcal{A}}x, y\rangle_{\mathcal{A}}).$ 
\end{theorem}

\begin{proof}
	Let $y\in \mathcal{H}$. Since $\theta$ is surjective,  there exists $x\in \mathcal{H}$ such that $\theta x=y$, and we have 
	\begin{equation*}
		A\langle x,x\rangle_{\mathcal{A}} A^{\ast}\leq \sum_{i\in I}\langle \Lambda_{i}x,\Lambda_{i}x\rangle_{\mathcal{A}}\leq B\langle x,x\rangle_{\mathcal{A}} B^{\ast}.
	\end{equation*}
	Thus 
	\begin{equation*}
		\phi(A\langle x,x\rangle_{\mathcal{A}} A^{\ast})\leq \phi\big( \sum_{i\in I}\langle \Lambda_{i}x,\Lambda_{i}x\rangle_{\mathcal{A}}\big)\leq \phi(B\langle x,x\rangle_{\mathcal{A}} B^{\ast}).
	\end{equation*}
	By definition of $\ast$-homomorphism, we have 
	\begin{equation*}
		\phi(A)\phi(\langle x,x\rangle_{\mathcal{A}} )\phi(A^{\ast})\leq \sum_{i\in I}\phi\big(\langle \Lambda_{i}x,\Lambda_{i}x\rangle_{\mathcal{A}}\big)\leq\phi(B)\phi(\langle x,x\rangle_{\mathcal{A}})\phi(B^{\ast}).
	\end{equation*}
	By the relation betwen $\theta$ and $\phi$, we get 
	\begin{equation*}
		\phi(A)\langle y,y\rangle_{\mathcal{B}}\phi(A)^{\ast}\leq \sum_{i\in I}\langle \Lambda_{i}y,\Lambda_{i}y\rangle_{\mathcal{B}}\leq  \phi(B)\langle y,y\rangle_{\mathcal{B}}\phi(B)^{\ast}.
	\end{equation*}
	On the other hand, we have
	\begin{align*}
		\phi(\langle S_{\mathcal{A}}x, y\rangle_{\mathcal{A}})&=\phi(\langle \sum_{i\in I}\Lambda_{i}^{\ast}\Lambda_{i}x,y\rangle_{\mathcal{A}})\\&=\sum_{i\in I}\phi(\langle \Lambda_{i}x,\Lambda_{i}y\rangle_{\mathcal{A}})\\&=\sum_{i\in I}\langle \Lambda_{i}\theta x,\Lambda_{i}\theta y\rangle_{\mathcal{B}}\\&=\langle \sum_{i\in I}\Lambda_{i}^{\ast}\Lambda_{i}\theta x, \theta y\rangle_{\mathcal{B}}\\&=\langle S_{\mathcal{B}}\theta x,\theta y\rangle_{\mathcal{B}}.
	\end{align*}
	This completes the proof.
\end{proof}

\begin{theorem}
	Let $\{\Lambda_{i}\in End_{\mathcal{A}}^{\ast}(\mathcal{H},V_{i}): i\in I\}$ be a  $\ast$-$g$-frame for $\mathcal{H}$ with lower and upper bounds $A$ and $B$, respectively. Let $\theta\in End_{\mathcal{A}}^{\ast}(\mathcal{H})$ be injective and have a closed range. Then $\{\theta \Lambda_{i}\}_{i\in I}$ is a  $\ast$-$g$-frame for $\mathcal{H}$.
\end{theorem}
\begin{proof}
	We have 
	\begin{equation*}
		A\langle x,x\rangle A^{\ast}\leq \sum_{i\in I}\langle \Lambda_{i}x,\Lambda_{i}x\rangle \leq B\langle x,x\rangle B^{\ast},\qquad\forall x\in\mathcal{H}.
	\end{equation*}
	Then for each $x\in\mathcal{H}$ 
	\begin{equation}\label{eq13}
		\sum_{i\in I}\langle \theta\Lambda_{i}x,\theta\Lambda_{i}x\rangle \leq ||\theta ||^{2}B\langle x,x\rangle B^{\ast}\leq (||\theta ||B)\langle x,x\rangle (||\theta||B)^{\ast}.
	\end{equation}
	By Lemma \ref{3}, we have for each $x\in\mathcal{H}$
	\begin{equation*}
		||(\theta^{\ast}\theta)^{-1}||^{-1}\langle \Lambda_{i}x,\Lambda_{i}x\rangle \leq \langle \theta\Lambda_{i}x,\theta\Lambda_{i}x\rangle
	\end{equation*}
	and $||\theta^{-1}||^{-2}\leq ||(\theta^{\ast}\theta)^{-1}||^{-1}$. Thus  
	\begin{equation}\label{eq14}
		||\theta^{-1}||^{-1}A\langle x,x\rangle (||\theta^{-1}||^{-1}A)^{\ast}\leq \sum_{i\in I}\langle \theta\Lambda_{i}x,\theta\Lambda_{i}x\rangle .
	\end{equation}
	From \eqref{eq13} and \eqref{eq14}, we have for each $x\in \mathcal{H}$
	\begin{eqnarray}
		||\theta^{-1}||^{-1}A\langle x,x\rangle (||\theta^{-1}||^{-1}A)^{\ast} & \leq & \sum_{i\in I}\langle \theta\Lambda_{i}x,\theta\Lambda_{i}x\rangle \\ & \leq &  ||\theta ||^{2}B\langle x,x\rangle B^{\ast}\\ &\leq & (||\theta ||B)\langle x,x\rangle (||\theta||B)^{\ast}.
	\end{eqnarray}
	We conclude that $\{\theta\Lambda_{i}\}_{i\in I}$ is a continuous $\ast$-$g$-frame for $\mathcal{H}$.
\end{proof}
\section{$K$-Frames}
\begin{definition} \cite{Gav}
	Let $K\in End_{\mathcal{A}}^{\ast}(\mathcal{H})$. A family $\{x_{i}\}_{i\in I}$ of elements of $\mathcal{H}$ is a $K$-frame for $ \mathcal{H} $, if there exist two positive constants $A$, $B$, such that for all $x\in\mathcal{H}$,
	\begin{equation}\label{11}
	A\langle K^{\ast}x,K^{\ast}x\rangle_{\mathcal{A}}\leq\sum_{i\in I}\langle x,x_{i}\rangle_{\mathcal{A}}\langle x_{i},x\rangle_{\mathcal{A}}\leq B\langle x,x\rangle_{\mathcal{A}}.
	\end{equation}
	The numbers $A$ and $B$ are called lower and upper bounds of the $K$-frame, respectively.
\end{definition}
The following theorem give a condition for getting a frame from a $K$-frame.
\begin{theorem}
	Let $\{x_{i}\}_{i\in I}$ be a $K$-frame for $\mathcal{H}$ with bounds $A, B>0$. If the operator $K$ is surjective, then $\{x_{i}\}_{i\in I}$ is a frame for $\mathcal{H}$. 
\end{theorem}
\begin{Prop}
	A Bessel sequence $\{x_{i}\}_{i\in I}$ of $\mathcal{H}$ is a $K$-frame with bounds $A, B>0$ if and only if $S\geq AKK^{\ast}$, where $S$ is the frame operator for $\{x_{i}\}_{i\in I}$.
\end{Prop}

\section{$K$-g-Frames}
\begin{definition} \cite{Xiang}
	Let $K\in End_{\mathcal{A}}^{\ast}(\mathcal{H})$ and $\Lambda_{i}\in End_{\mathcal{A}}^{\ast}(\mathcal{H},V_{i})$ for all $i\in I$, then $\{\Lambda_{i}\}_{i\in I}$ is said to be a $K$-g-frame for $\mathcal{H}$ with respect to $\{V_{i}\}_{i\in I}$ if there exist two constants $C, D > 0$ such that
	\begin{equation}
	C\langle K^{\ast}x,K^{\ast}x\rangle_{\mathcal{A}}\leq\sum_{i\in I}\langle \Lambda_{i}x,\Lambda_{i}x\rangle_{\mathcal{A}}\leq D\langle x,x\rangle_{\mathcal{A}}, \forall x\in\mathcal{H}.
	\end{equation}
	The numbers $C$ and $D$ are called $K$-g-frame bounds. Particularly, if 
	\begin{equation}
	C\langle K^{\ast}x,K^{\ast}x\rangle =\sum_{i\in I}\langle \Lambda_{i}x,\Lambda_{i}x\rangle, \forall x\in\mathcal{H}.
	\end{equation}
	The $K$-g-frame is $C$-tight.
\end{definition}
\begin{Ex}
	Let $l^{\infty}$ be the set of all bounded complex-valued sequences. For any $u=\{u_{j}\}_{j\in\mathbb{N}}, v=\{v_{j}\}_{j\in\mathbb{N}}\in l^{\infty}$, we define$$
	uv=\{u_{j}v_{j}\}_{j\in\mathbb{N}}, u^{\ast}=\{\bar{u_{j}}\}_{j\in\mathbb{N}}, \|u\|=\sup_{j\in\mathbb{N}}|u_{j}|.$$
	Then $\mathcal{A}=\{l^{\infty}, \|.\|\}$ is a $\mathbb{C}^{\ast}$-algebra.
	
	Let $\mathcal{H}=C_{0}$ be the set of all sequences converging to zero. For any $u, v\in\mathcal{H}$ we define$$\langle u,v\rangle=uv^{\ast}=\{u_{j}\bar{u_{j}}\}_{j\in\mathbb{N}}.$$
	Then $\mathcal{H}$ is a Hilbert $\mathcal{A}$-module.
	
	Now let $\{e_{j}\}_{j\in\mathbb{N}}$ be the standard orthonormal basis of $\mathcal{H}$. For each $j\in\mathbb{N}$ define the adjointable operator$$\Lambda_{j}:\mathcal{H}\rightarrow\overline{span}\{e_{j}\}, \;\;\Lambda_{j}x=\langle x,e_{j}\rangle e_{j},$$
	then for every $x\in\mathcal{H}$ we have $\sum_{j\in\mathbb{N}}\langle \Lambda_{j}x,\Lambda_{j}x\rangle=\langle x,x\rangle.$
	Fix $N\in\mathbb{N}^{\ast}$ and define$$K:\mathcal{H}\rightarrow\mathcal{H}, \;\;Ke_{j}=\begin{cases}
	je_{j}&\text{if $j\leq N$,}\\0&\text{if $j>N$.}
	\end{cases}$$
	It is easy to check that $K$ is adjointable and satisfies
	$$K^{\ast}e_{j}=\begin{cases}
	je_{j}&\text{if $j\leq N$,}\\0&\text{if $j>N$.}
	\end{cases}$$
	For any $x\in\mathcal{H}$ we have
	$$\dfrac{1}{N^{2}}\langle K^{\ast}x,K^{\ast}x\rangle\leq\sum_{j\in\mathbb{N}}\langle \Lambda_{j}x,\Lambda_{j}x\rangle=\langle x,x\rangle.$$
	This shows that $\{\Lambda_{j}\}_{j\in\mathbb{N}}$ is a $K$-g-frame with bounds $\dfrac{1}{N^{2}}, 1$.
\end{Ex}
\begin{Rem}
	If $K\in End_{\mathcal{A}}^{\ast}(\mathcal{H})$ is a surjective operator, then every K-g-frame for $\mathcal{H}$ with respect to $\{V_{i}\}_{i\in I}$. is a g-frame.
\end{Rem}

\begin{theorem}
	Let  $\{\Lambda_{i}\}_{i\in I}$ be a $K-g$ frame in Hilbert $ \mathcal{A}$ module $ \mathcal{H}$. and $ T \in  End^{*}_\mathcal{A}(\mathcal{H})$ such that $\mathcal{R}(T)\subset \mathcal{R}(K) $.  Then  $\{\Lambda_{i}\}_{i\in I}$ is a $T-g$-frame in Hilbert $ \mathcal{A}$-module $ \mathcal{H}$. 
	\begin{proof}
		Suppose that $C$  is a lower frame bound of  $\{\Lambda_{i}\}_{i\in I}$. Using Lemma \ref{lemm1. 2},  there exists $\alpha> 0$ such that $TT^{*}\leq\alpha^{2}KK^{*}$.  Now,  for each $x\in\mathcal{H} $.\\ We have $\langle TT^{*}x, x\rangle_{\mathcal{A}}\leq\alpha^{2} \langle KK^{*}x, x\rangle_{\mathcal{A}}. $\\
		 So
		\begin{align*} 
			\dfrac{C}{\alpha^{2}}\langle T^{*} x,  T^{*} x \rangle_{\mathcal{A}} &\leq C\langle K^{*} x,  K^{*} x \rangle_{\mathcal{A}}\\ &\leq \sum_{i\in I}\langle\Lambda_{i}x,  \Lambda_{i}x\rangle_{\mathcal{A}}\\&\leq D\langle x,x \rangle_{\mathcal{A}}.  
		\end{align*}
	\end{proof}   
\end{theorem}
\begin{theorem}Let  $\{\Lambda_{i}\}_{i\in I}$ be a  $K-g$- frame in Hilbert $ \mathcal{A}$-module $ \mathcal{H}$. Assume that $K$ has  a closed range and $ T \in  End^{*}_\mathcal{A}(\mathcal{H})$ such that  $\mathcal{R} (T^{*})\subset \mathcal{R}(K) $   
	Then $\{{\Lambda_{i}T^{*}}\}_{i\in{I}} $ is a  $K-g$- frame for $\mathcal{R}(T)$  if and only if there exists $\delta> 0$ such that for each $ x\in\mathcal{R}(T)$,  
	$$\lVert T^{*}x\lVert \geqslant \delta\lVert K^{*}x\lVert. $$
\end{theorem}
\begin{proof}
	Suppose that  $\{{\Lambda_{i}T^{*}}\}_{i\in{I}} $ is a  $K$-$g$-frame in Hilbert $ \mathcal{A}$ module $ \mathcal{H}$
	with a lower frame bound $E >0$.  If $F$  is an upper frame bound of  $\{\Lambda_{i}\}_{i\in I}$ then  for each $ x\in\mathcal{R}(T)$,  we have 
	$$ E\langle K^{*} x,  K^{*} x \rangle_{\mathcal{A}}\leq \sum_{i\in I}\langle\Lambda_{i}T^{*}x,  \Lambda_{i}T^{*}x\rangle_{\mathcal{A}}$$ 
	thus $$ E\langle K^{*} x,  K^{*} x \rangle_{\mathcal{A}}\leq \sum_{i\in I}\langle\Lambda_{i}T^{*} x,  \Lambda_{i}T^{*}x\rangle_{\mathcal{A}}\leq F\langle T^{*} x,   T^{*} x \rangle_{\mathcal{A}},  $$  
	Therefore $$ E \lVert \langle K^{*} x,  K^{*} x \rangle_{\mathcal{A}} \lVert \leq  \lVert \sum_{i\in I}\langle\Lambda_{i}T^{*} x,  \Lambda_{i}T^{*}x\rangle_{\mathcal{A}} \lVert \leq F \lVert \langle T^{*} x,T^{*} x \rangle_{\mathcal{A}} \lVert$$ thus $E \lVert  K^{*} x \lVert^{2} \leq F\lVert  T^{*} x \lVert^{2} $. 
	so $\sqrt{\frac{E}{F}} \lVert  K^{*} x \lVert \leq  \lVert  T^{*} x \lVert ,$
	for the opposite implication,   for each $x\in H$, we have $$ \lVert T^{*} x \lVert = \lVert( K^{\dagger})^{*}K^{*} T^{*}x \lVert \leq \lVert( K^{\dagger}) \lVert \lVert K^{*} U^{*}x \lVert. $$
	Therefore,   if $E$ is a lower frame bound of  $\{\Lambda_{i}\}_{i\in I}$,   we have 
	\begin{align*}
		E\delta^{2}\lVert K^{\dagger} \lVert^{-2} \langle  K^{*}x, K^{*}x\rangle &\leq E \lVert K^{\dagger} \lVert^{-2} \langle  T^{*}x,  T^{*}x \rangle \\
		& \leq  E \lVert K^{*}T^{*}x \lVert^{2}\\ 
		&\leq \sum_{i\in I}\langle\Lambda_{i}T^{*}x,  \Lambda_{i}T^{*}x\rangle_{\mathcal{A}}. 
	\end{align*} 
	For the upper bound,  it is clear that $$ \sum_{i\in I}\langle\Lambda_{i}T^{*}x,  \Lambda_{i}T^{*}x\rangle_{\mathcal{A}}  \leq F \langle T^{*} x,   T^{*}x \rangle_{\mathcal{A}} \leq F\lVert T \lVert^{2} \langle x,x \rangle_{\mathcal{A}} .$$
	So,   $({\Lambda_{i}T^{*}})_{i\in{I}} $ is a  $K-g$-frame in Hilbert $ \mathcal{A}$-module $ \mathcal{H}$ with frame bounds $E\delta^{2}\lVert K^{\dagger} \lVert^{-2}$ and $F \lVert T \lVert^{2}$ .
\end{proof}
\begin{theorem}
	Let  $\{\Lambda_{i}\}_{i\in I}$ be a  $K-g$-frame in Hilbert $ \mathcal{A}$-module $ \mathcal{H} $ and the operator $K$ has a dense rang.  Assume that  $ T\in  End^{*}_\mathcal{A}(\mathcal{H}) $ has a closed range and $ T$ and $T^{*}$ commute  .  
	If $\{{\Lambda_{i}T^{*}}\}_{i\in{I}} $ and $\{{\Lambda_{i}T}\}_{i\in{I}}$ are   $K-g$- frame in Hilbert $ \mathcal{A}$- module $\mathcal{H} $,  then $T$ is invertible.
	\begin{proof}
		Suppose that  $\{{\Lambda_{i}T^{*}}\}_{i\in{I}}$ is a  $K-g$-frame in Hilbert $ \mathcal{A}$ module $ \mathcal{H}$
		with a lower frame bound $A_{1}$,  and  $B_{1}$.  Then for each $ x\in\mathcal{H}$,    $$ A_{1}\langle K^{*} x,  K^{*} x \rangle_{\mathcal{A}}\leq \sum_{i\in I}\langle\Lambda_{i }T^{*} x,  \Lambda_{i}T^{*}x\rangle_{\mathcal{A}}\leq B_{1}\langle x,x \rangle_{\mathcal{A}} .$$ 
		We have
		\begin{equation}
			\lVert  A_{1}\langle K^{*}x,  K^{*}x \rangle_{\mathcal{A}}\lVert \leq \lVert \sum_{i\in I}\langle\Lambda_{i }T^{*}x,  \Lambda_{i}T^{*}x\rangle_{\mathcal{A}} \lVert \leq \lVert  B_{1}\langle x,   x \rangle_{\mathcal{A}} \lVert , 
		\end{equation}
		hence,    $$ A_{1} \lVert K^{*}x \lVert^{2} \leq \lVert \sum_{i\in I}\langle\Lambda_{i }T^{*}x,  \Lambda_{i}T^{*}x\rangle \lVert \leq B_{1} \lVert x \lVert^{2}  .$$
		Since $K$ has a dense range,   $ K^{*}$ is injective.  
		Moreover,  $ \mathcal{R}(T) =( \ker T^{*})^{\perp} = H  $ so $T$ is surjective. 
		Suppose that  $\{{\Lambda_{i}T^{*}}\}_{i\in{I}} $ is a  $K-g$-frame in Hilbert $ \mathcal{A}$ module $ \mathcal{H}$
		with a lower frame bound $A_{2}$  and  $B_{2}$. Then,  for each $x	\in \mathcal{H}$,   $$ A_{2}\langle K^{*} x,  K^{*} x \rangle_{\mathcal{A}}\leq \sum_{i\in I}\langle\Lambda_{i }T^{*}x,  \Lambda_{i}T^{*}x\rangle_{\mathcal{A}}\leq B_{2}\langle x,x \rangle_{\mathcal{A}} $$ $$\lVert  A_{2}\langle K^{*}x,  K^{*}x \rangle_{\mathcal{A}}\lVert \leq \lVert \sum_{i\in I}\langle\Lambda_{i }T^{*}x,  \Lambda_{i}T^{*}x\rangle_{\mathcal{A}} \lVert \leq \lVert  B_{2}\langle x,x \rangle_{\mathcal{A}} \lVert $$  $$ A_{2} \lVert K^{*}x\lVert^{2} \leq \lVert \sum_{i\in I}\langle\Lambda_{i }T^{*}x,  \Lambda_{i}T^{*}x\rangle_{\mathcal{A}} \lVert \leq B_{2} \lVert x \lVert^{2}  .$$  
		Therefore $T$ is injective,   since $\ker U \subseteq \ker K^{*} .$ 
		Thus,   $T$ is an invertible operator. 
	\end{proof}
\end{theorem}
\begin{theorem}
	Let $\{\Lambda_{i}\}_{i\in I}$ be a  $K-g$-frame in Hilbert $ \mathcal{A}$- module ${\mathcal{H}} $  and   $ T \in  End^{*}_\mathcal{A}(\mathcal{H})$ be a co-isometry (i.e. $ TT^{*}= Id_{H}$) such that  $TK=KT $  Then  $\{\Lambda_{i}T^{*}\}_{i\in{I}} $ is a  $K-g$-frame in Hilbert $ \mathcal{A}$-module ${\mathcal{H}} $. 	
\end{theorem}
\begin{proof}
	Suppose  $\{\Lambda_{i}\}_{i\in I}$ be a  $K-g$- frame in Hilbert $ \mathcal{A}$-module ${\mathcal{H}} $ with a lower frame bound $A_{1}$.  and  $B_{1}$  for each $ x\in\mathcal{H}$,   we have $$ \sum_{i\in I}\langle\Lambda_{i }T^{*}x,  \Lambda_{i}T^{*}x\rangle_\mathcal{A} \leq   B_{1} \langle T^{*}x, T^{*}x \rangle_{\mathcal{A}}$$ hence, $$ \sum_{i\in I}\langle\Lambda_{i }T^{*}x,  \Lambda_{i}T^{*}x\rangle_ \mathcal{A} \leq B_{1} \lVert T^{*} \lVert^{2} \langle x, x \rangle_{\mathcal{A}} .$$ 
	So, $\{{\Lambda_{i}T^{*}}\}_{i\in{I}} $ is a  $g$-Bessel sequence.\\ For the lower bound,   we can write
	\begin{align*}
		   \sum_{i\in I}\langle\Lambda_{i }T^{*} x,  \Lambda_{i}T^{*}x\rangle_\mathcal{A}
		&\geqslant  A_{1} \langle  K^{*}T^{*}x,   K^{*}T^{*}x  \rangle_\mathcal{A}\\
		&=A_{1} \langle (TK)^{*}x,  (TK)^{*}x  \rangle_\mathcal{A}\\
		&=A_{1} \langle (KT)^{*}x,  (KT)^{*}x  \rangle_\mathcal{A}\\
		&= A_{1} \langle  T^{*}K^{*}x,   T^{*}K^{*}x \rangle_\mathcal{A}\\&= A_{1} \langle TT^{*}K^{*}x,K^{*}x  \rangle_\mathcal{A}\\
		&=A_{1} \langle  K^{*}x,   K^{*}x  \rangle_\mathcal{A}.
	\end{align*}
\end{proof}
\begin{theorem}
	Let $ \Lambda:=\{  \Lambda_{i}\in {End^{*}_{\mathcal{A}}(\mathcal{H},\mathcal{V}_{i})}\}_{i\in I}$ and $ \circleddash:=\{ \circleddash_{i}\in{End^{*}_{\mathcal{A}}(\mathcal{H},\mathcal{V}_{i})}\}_{i\in I}$ be tow  $K-g$- Bessel sequences  in Hilbert $ \mathcal{A}$- module ${\mathcal{H}}$ with bounds $B_{\Lambda}$ and $B_{\circleddash}$  respectively. Suppose that $T_{\Lambda}$ and $T_{\circleddash}$ are their synthesis operators such that $T_{\circleddash, }T_{\Lambda}^{*}=K^{*} $.  Then $ \Lambda $ and $ \circleddash$ are  $K$ and $ K^{*}-g$-frames,  respectively.  
\end{theorem}
\begin{proof}  For each $ x\in\mathcal{H}$,   we have
	\begin{align*}
		\lVert K^{*}x \lVert^{2} & = \lVert \langle  K^{*}x,    K^{*}x \rangle_\mathcal{A}\lVert  \\
		&= \lVert \langle T_{\circleddash} T_{\Lambda}^{*}x ,    K^{*}x \rangle_\mathcal{A}\lVert \\ 
		&\leq  \lVert T_{\Lambda}^{*}x \lVert\lVert T_{\circleddash, }^{*}K^{*}x \lVert \\
		&\leq(\sum_{i\in I}\langle\Lambda_{i } x,  \Lambda_{i}x\rangle_\mathcal{A})^{1/2}B_{\circleddash} \lVert\langle  K^{*}x,   K^{*}x\rangle_\mathcal{A}\lVert.
	\end{align*}
	So, $$\lVert \langle  K^{*}x,    K^{*}x \rangle_\mathcal{A}\lVert \leq \sum_{i\in I}\langle\Lambda_{i }x,  \Lambda_{i}x\rangle_\mathcal{A}B_{\circleddash}  .$$
\end{proof}
\section{$\ast$-$K$-Frames}
\begin{definition} \cite{Dast}
	Let $K\in End_{\mathcal{A}}^{\ast}(\mathcal{H})$. A family $\{x_{i}\}_{i\in I}$ of elements of $\mathcal{H}$ is an $\ast$-K-frame for $ \mathcal{H} $, if there	exist strictly nonzero elements $A$, $B$ in $\mathcal{A}$, such that for all $x\in\mathcal{H}$,
	\begin{equation}
	A\langle K^{\ast}x,K^{\ast}x\rangle_{\mathcal{A}}A^{\ast}\leq\sum_{i\in I}\langle x,x_{i}\rangle_{\mathcal{A}}\langle x_{i},x\rangle_{\mathcal{A}}\leq B\langle x,x\rangle_{\mathcal{A}}B^{\ast}.
	\end{equation}
	The elements $A$ and $B$ are called lower and upper bound of the $\ast$-$K$-frame, respectively.
\end{definition}
\begin{Rem}
	Every $\ast$-frame is a $\ast$-$K$-frame.
\end{Rem}
\section{$\ast$-$K$-g-Frames in Hilbert A-modules}
\begin{definition} \label{our frame} \cite{Ros2}
	Let $K\in End_{\mathcal{A}}^{\ast}(\mathcal{H})$. We call a sequence $\{\Lambda_{i}\in End_{A}^{\ast}(\mathcal{H},\mathcal{H}_{i}):i\in I \}$ a $\ast$-K-g-frame in Hilbert $\mathcal{A}$-module $\mathcal{H}$ with respect to $\{\mathcal{H}_{i}:i\in I \}$ if there exist strictly nonzero elements $A$, $B$ in $\mathcal{A}$ such that 
	\begin{equation}
	A\langle K^{\ast}x,K^{\ast}x\rangle_{\mathcal{A}} A^{\ast}\leq\sum_{i\in I}\langle \Lambda_{i}x,\Lambda_{i}x\rangle_{\mathcal{A}}\leq B\langle x,x\rangle_{\mathcal{A}} B^{\ast}, \forall x\in\mathcal{H}.
	\end{equation}
	The numbers $A$ and $B$ are called lower and upper bounds of the $\ast$-K-g-frame, respectively. If 
	\begin{equation}
	A\langle K^{\ast}x,K^{\ast}x\rangle A^{\ast}=\sum_{i\in I}\langle \Lambda_{i}x,\Lambda_{i}x\rangle, \forall x\in\mathcal{H}.
	\end{equation}
	The $\ast$-K-g-frame is $A$-tight.
\end{definition}
\begin{Rem} \cite{Ros2} \label{rem3.2}
	\begin{enumerate}
		\item Every $\ast$-g-frame for $\mathcal{H}$ with respect to $\{\mathcal{H}_{i}:i\in I \}$ is an $\ast$-K-g-frame, for any $K\in End_{\mathcal{A}}^{\ast}(\mathcal{H})$: $K\neq0$.
		\item If $K\in End_{\mathcal{A}}^{\ast}(\mathcal{H})$ is a surjective operator, then every $\ast$-K-g-frame for $\mathcal{H}$ with respect to $\{\mathcal{H}_{i}:i\in I \}$ is a $\ast$-g-frame.
	\end{enumerate}
\end{Rem}
\begin{Ex} \cite{Ros2}
	Let $\mathcal{H}$ be a finitely or countably generated Hilbert $\mathcal{A}$-module. Let $K\in End_{\mathcal{A}}^{\ast}(\mathcal{H})$: $K\neq0$. Let $\mathcal{A}$ be a Hilbert $\mathcal{A}$-module over itself with the inner product $\langle a,b\rangle=ab^{\ast}$. 	Let $\{x_{i}\}_{i\in I}$ be an $\ast$-frame for $\mathcal{H}$ with bounds $A$ and $B$, respectively. For each $i\in I$, we define $\Lambda_{i}:\mathcal{H}\to\mathcal{A}$ by $\Lambda_{i}x=\langle x,x_{i}\rangle,\;\; \forall x\in\mathcal{H}$. $\Lambda_{i}$ is adjointable and $\Lambda_{i}^{\ast}a=ax_{i}$ for each $a\in\mathcal{A}$. And we have 
	\begin{equation*}
	A\langle x,x\rangle A^{\ast}\leq\sum_{i\in I}\langle x,x_{i}\rangle\langle x_{i},x\rangle\leq B\langle x,x\rangle B^{\ast}, \forall x\in\mathcal{H}.
	\end{equation*}
	Or \begin{displaymath}
	\langle K^{\ast}x,K^{\ast}x\rangle\leq\|K\|^{2}\langle x,x\rangle, \forall x\in\mathcal{H}. 
	\end{displaymath}
	Then
	\begin{equation*}
	\|K\|^{-1}A\langle K^{\ast}x, K^{\ast}x\rangle (\|K\|^{-1}A)^{\ast}\leq\sum_{i\in I}\langle \Lambda_{i}x,\Lambda_{i}x\rangle\leq B\langle x,x\rangle B^{\ast}, \forall x\in\mathcal{H}.
	\end{equation*}
	So $\{\Lambda_{i}\}_{i\in I}$ is an $\ast$-K-g-frame for $\mathcal{H}$ with bounds $\|K\|^{-1}A$ and $B$, respectively.
\end{Ex}
Let $\{\Lambda_{i}\}_{i\in I}$ be an $\ast$-K-g-frame in $\mathcal{H}$ with respect to $\{\mathcal{H}_{i}:i\in I \}$. Define an operator $T:\mathcal{H}\to\oplus_{i\in I}\mathcal{H}_{i}$ by $Tx=\{\Lambda_{i}x\}_{i}, \forall x\in\mathcal{H}$, then $T$ is called the analysis operator. So it's adjoint operator is $T^{\ast}:\oplus_{i\in I}\mathcal{H}_{i}\to\mathcal{H}$ given by $T^{\ast}(\{x_{i}\}_{i})=\sum_{i\in I}\Lambda_{i}^{\ast}x_{i}, \forall\{x_{i}\}_{i}\in\oplus_{i\in I}\mathcal{H}_{i}$. The operator $T^{\ast}$ is called the synthesis operator. By composing $T$ and $T^{\ast}$, the frame operator $S:\mathcal{H}\to\mathcal{H}$ is given by $Sx=T^{\ast}Tx=\sum_{i\in I}\Lambda_{i}^{\ast}\Lambda_{i}x$.

Note that $S$ need not be invertible in general. But under some condition $S$ will be invertible.
\begin{theorem} \cite{Ros2}
	Let $K\in End_{\mathcal{A}}^{\ast}(\mathcal{H})$ be a surjective operator. If $\{\Lambda_{i}\}_{i\in I}$ is an $\ast$-K-g-frame in $\mathcal{H}$ with respect to $\{\mathcal{H}_{i}:i\in I \}$, then the frame operator $S$ is positive, invertible and adjointable. Moreover we have the reconstruction formula, $x=\sum_{i\in I}\Lambda_{i}^{\ast}\Lambda_{i}S^{-1}x, \forall x\in\mathcal{H}$.
\end{theorem}
\begin{proof}
	Result of $(2)$ in Remark \ref{rem3.2} and Theorem $3.8$ in \cite{A}.
\end{proof}
Let $K\in End_{\mathcal{A}}^{\ast}(\mathcal{H})$, in the following theorem using an $\ast$-g-frame we constructed an $\ast$-K-g-frame.
\begin{theorem} \cite{Ros2} \label{Th3.5}
	Let $K\in End_{\mathcal{A}}^{\ast}(\mathcal{H})$ and $\{\Lambda_{i}\}_{i\in I}$ be an $\ast$-g-frame in $\mathcal{H}$ with respect to $\{\mathcal{H}_{i}:i\in I \}$ with bounds $A$, $B$. Then $\{\Lambda_{i}K\}_{i\in I}$ is an $\ast$-$K^{\ast}$-g-frame in $\mathcal{H}$ with respect to $\{\mathcal{H}_{i}:i\in I \}$ with bounds $A$, $\|K\|B$. The frame operator of $\{\Lambda_{i}K\}_{i\in I}$ is $S^{'}=K^{\ast}SK$, where $S$ is the frame operator of $\{\Lambda_{i}\}_{i\in I}$.
\end{theorem}
\begin{proof}
	Form
	\begin{equation*}
	A\langle x,x\rangle_{\mathcal{A}} A^{\ast}\leq\sum_{i\in I}\langle \Lambda_{i}x,\Lambda_{i}x\rangle_{\mathcal{A}}\leq B\langle x,x\rangle_{\mathcal{A}} B^{\ast}, \forall x\in\mathcal{H}.
	\end{equation*}
	We get for all $x\in\mathcal{H}$,
	\begin{equation*}
	A\langle Kx,Kx\rangle_{\mathcal{A}} A^{\ast}\leq\sum_{i\in I}\langle \Lambda_{i}Kx,\Lambda_{i}Kx\rangle_{\mathcal{A}}\leq B\langle Kx,Kx\rangle_{\mathcal{A}} B^{\ast}\leq \|K\|B\langle x,x\rangle_{\mathcal{A}} (\|K\|B)^{\ast}.
	\end{equation*}
	Then $\{\Lambda_{i}K\}_{i\in I}$ is an $\ast$-$K^{\ast}$-g-frame in $\mathcal{H}$ with respect to $\{\mathcal{H}_{i}:i\in I \}$ with bounds $A$, $\|K\|B$.
	
	By definition of $S$,we have $SKx=\sum_{i\in I}\Lambda_{i}^{\ast}\Lambda_{i}Kx$. Then $$K^{\ast}SKx=K^{\ast}\sum_{i\in I}\Lambda_{i}^{\ast}\Lambda_{i}Kx=\sum_{i\in I}K^{\ast}\Lambda_{i}^{\ast}\Lambda_{i}Kx.$$
	Hence $S^{'}=K^{\ast}SK$.
\end{proof}
\begin{corollary} \cite{Ros2}
	Let $K\in End_{\mathcal{A}}^{\ast}(\mathcal{H})$ and $\{\Lambda_{i}\}_{i\in I}$ be an $\ast$-g-frame. Then $\{\Lambda_{i}S^{-1}K\}_{i\in I}$ is an $\ast$-$K^{\ast}$-g-frame, where $S$ is the frame operator of $\{\Lambda_{i}\}_{i\in I}$.
\end{corollary}
\begin{proof}
	Result of the Theorem \ref{Th3.5} for the $\ast$-g-frame $\{\Lambda_{i}S^{-1}\}_{i\in I}$.	
\end{proof}

\section{Some consequences}
\begin{itemize}
	\item If $A, B\in\mathbf{C}$ and $K=I$ in Definition \ref{our frame} we find the definition of the g-frame. 
	\item If $A, B\in\mathbf{C}$ in Definition \ref{our frame} we find the definition of the $K$-g-frame.
	\item
	For $K = I$ in Definition \ref{our frame} we find the definition of the $\ast$-g-frame. 
\end{itemize}
\section{Operator frame}
\begin{definition}
	A family of adjointable operators $\{T_{i}\}_{i\in\mathbb{J}}$ on a Hilbert $\mathcal{A}$-module $\mathcal{H}$ over a unital $C^{\ast}$-algebra is said to be an operator frame for $End_{\mathcal{A}}^{\ast}(\mathcal{H})$, if there exists positive constants $A, B > 0$ such that 
	\begin{equation}\label{eq3}
		A\langle x,x\rangle\leq\sum_{i\in\mathbb{J}}\langle T_{i}x,T_{i}x\rangle\leq B\langle x,x\rangle, \forall x\in\mathcal{H}.
	\end{equation}
	The numbers $A$ and $B$ are called lower and upper bounds of the operator frame, respectively. If $A=B=\lambda$, the operator frame is $\lambda$-tight. If $A = B = 1$, it is called a normalized tight operator frame or a Parseval operator frame. If only upper inequality of \eqref{eq3} hold, then $\{T_{i}\}_{i\in\mathbb{J}}$ is called an operator Bessel sequence for $End_{\mathcal{A}}^{\ast}(\mathcal{H})$.\\
	If the sum in the middle of \eqref{eq3} is convergent in norm, the operator frame is called standard. 
\end{definition}
Throughout the paper, series like \eqref{eq3} are assumed to be convergent in the norm sense.
\begin{Ex}
	Let $\mathcal{A}$ be a Hilbert $C^{\ast}$-module over itself with the inner product $\langle a,b\rangle=ab^{\ast}$.
	Let $\{x_{i}\}_{i\in I}$ be a frame for $\mathcal{A}$ with bounds $A$ and $B$, respectively. For each $i\in I$, we define $T_{i}:\mathcal{A}\to\mathcal{A}$ by $T_{i}x=\langle x,x_{i}\rangle,\;\; \forall x\in\mathcal{A}$. $T_{i}$ is adjointable and $T_{i}^{\ast}a=ax_{i}$ for each $a\in\mathcal{A}$. And we have 
	\begin{equation*}
		A\langle x,x\rangle\leq\sum_{i\in I}\langle x,x_{i}\rangle\langle x_{i},x\rangle\leq B\langle x,x\rangle, \forall x\in\mathcal{A}.
	\end{equation*}
	Then
	\begin{equation*}
		A\langle x,x\rangle\leq\sum_{i\in I}\langle T_{i}x,T_{i}x\rangle\leq B\langle x,x\rangle, \forall x\in\mathcal{A}.
	\end{equation*}
	So $\{T_{i}\}_{i\in I}$ is an operator frame in $\mathcal{A}$ with bounds $A$ and $B$, respectively.
\end{Ex}
\begin{Ex}
	Let $\mathcal{H}$ and $\mathcal{K}$ be separable Hilbert spaces and let $B(\mathcal{H,K)}$ be the set of all bounded linear operators from $\mathcal{H}$ into $\mathcal{K}$. $B(\mathcal{H,K)}$ is a Hilbert $B(\mathcal{K})$-module with a $B(\mathcal{K})$-valued inner product $\langle S, T\rangle= ST^{\ast}$ for all $S, T\in B(\mathcal{H,K)}$, and with a linear operation of $B(\mathcal{K})$ on $B(\mathcal{H,K)}$ by composition of operators.
	\\
	Let $\mathbb{J}=\mathbb{N}$ and fix $(a_{i})_{i\in\mathbb{N}}\in l^{2}(\mathbf{C})$. Define: $$T_{i}(X)=a_{i}X, \forall X\in B(\mathcal{H,K)}, \forall i\in\mathbb{N}.$$ 
	We have 
	$$\sum_{i\in\mathbb{N}}\langle T_{i}x,T_{i}x\rangle=\sum_{i\in\mathbb{N}}|a_{i}|^{2}\langle X, X\rangle, \forall X\in B(\mathcal{H,K)}.$$
	\\
	Then $\{T_{i}\}_{i\in\mathbb{N}}$ is $\sum_{i\in\mathbb{N}}|a_{i}|^{2}$-tight operator frame.
\end{Ex}
\begin{Ex}
	Let $\{W_{i}\}_{i\in\mathbb{J}}$ be a frame of submodules with respect to $ \{v_{i}\}_{i\in\mathbb{J}} $ for $\mathcal{H}$. Put $T_{i}=v_{i}\pi_{W_{i}}, \forall i\in\mathbb{J}$, then we get a sequence of operators $\{T_{i}\}_{i\in\mathbb{J}}$. Then there exist constants $A, B >0$ such that:
	\begin{displaymath}
		A\langle x,x\rangle\leq\sum_{i\in\mathbb{J}}v_{i}^{2}\langle \pi_{W_{i}}x,\pi_{W_{i}}x\rangle\leq B\langle x,x\rangle, \forall x\in\mathcal{H}.
	\end{displaymath}
	So we have:
	\begin{displaymath}
		A\langle x,x\rangle\leq\sum_{i\in\mathbb{J}}\langle T_{i}x,T_{i}x\rangle\leq B\langle x,x\rangle, \forall x\in\mathcal{H}.
	\end{displaymath}
	Thus, the sequence $\{T_{i}\}_{i\in\mathbb{J}}$ becomes an operator frame for $\mathcal{H}$.
\end{Ex}
With this example a frame of submodules can be viewed as a special case of operator frames.
\begin{theorem}
	Let $ \{T_i\}_{i \in I}$ be a operator frame for $End_{\mathcal{A}}^{\ast}(\mathcal{H})$ with bounds $\nu$ and $\delta$. If $\{R_i\}_{i \in I} \subset End_{\mathcal{A}}^{\ast}(\mathcal{H}) $ is a operator Bessel family with bound $\xi<\nu$, then $ \{T_i \mp R_i\}_{i \in I}$ is a  operator frame for $End_{\mathcal{A}}^{\ast}(\mathcal{H})$. 
	
\end{theorem}
\begin{proof}
	We just proof the case that $ \{T_w + R_w\}_{w \in \Omega}$ is a operator frame for $End_{\mathcal{A}}^{\ast}(\mathcal{H})$. \\
	On one hand, For each $x \in \mathcal{H}$, we have
	\begin{align*}
		\|\{(T_i + R_i)f\}_{i \in I}\| &=\|\sum_{i \in I}\langle (T_i + R_i)f,(T_i + R_i)f\rangle_{\mathcal{A}}\|^{\frac{1}{2}}\\
		&\leq \|\{T_if\}_{i \in I}\|+\|\{ R_if\}_{i \in I}\|\\
		&= \|\sum_{i \in I}\langle T_i f,T_i f\rangle_{\mathcal{A}} \|^{\frac{1}{2}} + \|\sum_{i \in I}\langle R_i f, R_i f\rangle_{\mathcal{A}}\|^{\frac{1}{2}}\\
		&\leq \sqrt{\delta}\|f\|+\sqrt{\xi}\|f\|.
	\end{align*}
	Hence
	\begin{equation}\label{haja1}
		\|\sum_{i \in I}\langle (T_i + R_i)f,(T_i + R_i)f\rangle_{\mathcal{A}} \|^{\frac{1}{2}}\leq (\sqrt{\delta}+\sqrt{\xi})\|f\|.
	\end{equation}
	One the other hand we have
	\begin{align*}
		\|\{(T_i + R_i)f\}_{i \in I}\|&=\|\sum_{i \in I}\langle (T_i + R_i)f,(T_i + R_i)f\rangle_{\mathcal{A}}\|^{\frac{1}{2}}\\
		&\geq \|\{T_if\}_{i \in I}\|-\|\{ R_if\}_{i \in I}\|\\
		&=\|\sum_{i \in I}\langle T_i f,T_i f\rangle_{\mathcal{A}} \|^{\frac{1}{2}}-\|\sum_{i \in I}\langle  R_{i}f, R_{i}f\rangle_{\mathcal{A}} \|^{\frac{1}{2}}\\
		&\geq \sqrt{\nu}\|f\|-\sqrt{\xi}\|f\|.
	\end{align*}
	Then
	\begin{equation}\label{haja2}
		\|\sum_{i \in I}\langle (T_i + R_i)f,(T_i + R_i)f\rangle_{\mathcal{A}} \|^{\frac{1}{2}}\geq  (\sqrt{\nu}-\sqrt{\xi})\|f\|.
	\end{equation}
	From (\ref{haja1}) and (\ref{haja2}) we get
	$$(\sqrt{\nu}-\sqrt{\xi})^2\|f\|^2\leq\|\sum_{i \in I}\langle (T_i + R_i)f,(T_i + R_i)f\rangle_{\mathcal{A}}\|\leq (\sqrt{\delta}+\sqrt{\xi})^2\|f\|^2 .$$
	Therefore $ \{T_i + R_i\}_{i \in I}$ is a operator frame for $End_{\mathcal{A}}^{\ast}(\mathcal{H})$.
\end{proof}
\begin{theorem}
	Let $ \{T_i \}_{i \in I}$ be a operator frame for $End_{\mathcal{A}}^{\ast}(\mathcal{H})$ with bounds $\nu$ and $\delta$ and let $ \{R_i\}_{i \in I} \subset End_{\mathcal{A}}^{\ast}(\mathcal{H})$. The following statements are equivalent:
	\begin{itemize}
		\item [(i)] $ \{R_i\}_{i \in I}$ is a operator frame for $End_{\mathcal{A}}^{\ast}(\mathcal{H})$ .
		\item [(ii)] There exists a constant $\xi>0$, such that for all x in $\mathcal{H} $
	\end{itemize}	
	we have 
	\begin{equation}\label{haja3}
		\scriptsize 
		\|\sum_{i \in I}\langle (T_i - R_i)f,(T_i - R_i)f\rangle_{\mathcal{A}}\|\leq \xi. min(\|\sum_{i \in I}\langle T_i f,T_i f\rangle_{\mathcal{A}}\|, \|\sum_{i \in I}\langle R_i f,R_i f\rangle_{\mathcal{A}} \|).
	\end{equation}
\end{theorem}	
\begin{proof}
	Suppose that $ \{R_i\}_{i \in I}$  is a operator frame for $End_{\mathcal{A}}^{\ast}(\mathcal{H})$ with bound $\eta$ and $\rho$. Then for all $f \in \mathcal{H} $ we have 
	\begin{align*}
		\|\{(T_i - R_i)f\}_{i \in I}\|&=\|\sum_{i \in I}\langle (T_i - R_i)f,(T_i - R_i)f\rangle_{\mathcal{A}}\|^{\frac{1}{2}}\\
		&\leq \|\{T_if\}_{i \in I}\|+\|\{ R_if\}_{i \in I}\|\\
		&=\|\langle T_i f,T_i f\rangle_{\mathcal{A}}\|^{\frac{1}{2}}+\|\sum_{i \in I}\langle R_if, R_if\rangle_{\mathcal{A}}\|^{\frac{1}{2}}\\
		&\leq \|\sum_{i \in I}\langle T_i f,T_i f\rangle_{\mathcal{A}} \|^{\frac{1}{2}}+ \sqrt{\rho}\|f\|\\
		&\leq  \|\sum_{i \in I}\langle T_i f,T_i f\rangle_{\mathcal{A}} \|^{\frac{1}{2}}+\sqrt{\frac{\rho}{\nu}}  \|\sum_{i \in I}\langle T_i f,T_i f\rangle_{\mathcal{A}} \|^{\frac{1}{2}}\\
		&= (1+\sqrt{\frac{\rho}{\nu}}) \| \sum_{i \in I}\langle T_i f,T_i f\rangle_{\mathcal{A}} \|^{\frac{1}{2}}.
	\end{align*}
	In the same way we have 
	$$\|\sum_{i \in I}\langle (T_i - R_i)f,(T_i - R_i)f\rangle_{\mathcal{A}}\|^{\frac{1}{2}}\leq \left(1+\sqrt{\frac{\delta}{\eta}}\right) \|\sum_{i \in I}\langle R_i f,R_i f\rangle_{\mathcal{A}} \|^{\frac{1}{2}}.$$
	For (\ref{haja3}), we take $\xi= min (1+\sqrt{\frac{\delta}{\eta}} ,1+\sqrt{\frac{\rho}{\nu}} )$.\\
	Now we assume that (\ref{haja3}) holds. For each $f \in \mathcal{H} $, we have: \\
	From ( \ref{haja3}), we have 
	$$	\|\sum_{i \in I}\langle (T_i - R_i)f,(T_i - R_i)f\rangle_{\mathcal{A}}\|\leq \xi \|\sum_{i \in I}\langle R_i f,R_i f\rangle_{\mathcal{A}}\|.$$
	Then
	$$\|\sum_{i \in I}\langle T_i f,T_i f\rangle_{\mathcal{A}}\|^{\frac{1}{2}} \leq \sqrt{\xi}\|\sum_{i \in I}\langle R_i f,R_i f\rangle_{\mathcal{A}}\|^{\frac{1}{2}}+\| \sum_{i \in I}\langle R_i f,R_i f\rangle_{\mathcal{A}} \|^{\frac{1}{2}}. $$
	Hence 
	\begin{equation} \label{haja4}
		\sqrt{\nu}\|f\| \leq\|\sum_{i \in I}\langle T_i f,T_i f\rangle_{\mathcal{A}} \|^{\frac{1}{2}}\leq (1+\sqrt{\xi})\| \sum_{i \in I}\langle R_i f,R_i f\rangle_{\mathcal{A}}\|^{\frac{1}{2}}.
	\end{equation}
	Also, we have 
	\begin{align*}
		\|\{ R_if\}_{i \in I}\|&=\|\sum_{i \in I}\langle R_i f,R_i f\rangle_{\mathcal{A}} \|^{\frac{1}{2}}\\
		&= \|\{ (R_if - T_if) + T_if\}_{i \in I}\|\\
		&\leq \|\{(T_i - R_i)f\}_{i \in I}\|+\|\{ T_if\}_{i \in I}\|\\  
		&= \|\sum_{i \in I}\langle (T_i - R_i)f,(T_i - R_i)f\rangle_{\mathcal{A}} \|^{\frac{1}{2}}+ \|\sum_{i \in I}\langle T_i f,T_if\rangle_{\mathcal{A}} \|^{\frac{1}{2}}\\    
	\end{align*}
	From (\ref{haja3}), we have 
	
	$$	\|\sum_{i \in I}\langle (T_i - R_i)f,(T_i - R_i)f\rangle_{\mathcal{A}}\|\leq \xi \|\sum_{i \in I}\langle T_i f,T_i f\rangle_{\mathcal{A}}\|.$$
	Then 
	
	$$\|\sum_{i \in I}\langle R_i f,R_i f\rangle_{\mathcal{A}}\|^{\frac{1}{2}} \leq (1+\sqrt{\xi})\|\sum_{i \in I}\langle T_i f,T_i f\rangle_{\mathcal{A}} \|^{\frac{1}{2}}.$$
	So, 
	\begin{equation}\label{haja5}
		\|\sum_{i \in I}\langle R_i f,R_i f\rangle_{\mathcal{A}}\|^{\frac{1}{2}} \leq (1+\sqrt{\xi})\sqrt{\delta} \|f\|.
	\end{equation}
	From (\ref{haja4})  and (\ref{haja5}) we give that 
	$$\frac{\nu}{(1+\sqrt{\xi})^2} \|f\|^2 \leq \|\sum_{i \in I}\langle R_i f,R_i f\rangle_{\mathcal{A}}\|\leq \delta (1+\sqrt{\xi})^2\|f\|^2.$$
	Therefore $ \{R_i\}_{i \in I}$ is a operator frame for $End_{\mathcal{A}}^{\ast}(\mathcal{H})$. 
\end{proof}
\begin{theorem}
	For $k=1,2,...,n$, let $\{T_{k,i}\} _{i \in I} \subset End_{\mathcal{A}}^{\ast}(\mathcal{H})$ be a operator frames with bounds  $A_k$ and  $B_k$. Let $\{R_{k,i}\} _{i \in I}\subset End_{\mathcal{A}}^{\ast}(\mathcal{H})$ and $L: l^2(\mathcal{H})\longrightarrow l^2(\mathcal{H}) $ be a bounded linear operator such that:
	\begin{equation*}
		L(\{\sum_{k=1}^{n}R_{k,i}x\} _{i \in I})=\{T_{p,i}x\} _{i \in I} \quad for some \quad p \in \{1,2,..,n\}.
	\end{equation*}
	If there exists a constant $\lambda >0$ such that for each $x\in \mathcal{H}$ and $ k=1,..,n$ we have,
	
	\begin{equation*}
		\| \sum_{i\in I}\langle (T_{k,i}-R_{k,i}) x,(T_{k,i}-R_{k,i}) x\rangle_{\mathcal{A}} \| \leq \lambda \|\sum_{i\in I}\langle T_{k,i} x,T_{k,i} x\rangle_{\mathcal{A}} \|.
	\end{equation*}
	Then $\{\sum_{k=1}^{n}R_{k,i}\} _{i \in I}$ is a operator frame for $End_{\mathcal{A}}^{\ast}(\mathcal{H})$.
\end{theorem}
\begin{proof}
	For all $x \in \mathcal{H}$, we have
	\begin{align*}
		\|\{\sum_{k=1}^{n}R_{k,i}x\} _{i \in I}\|&\leq \sum_{k=1}^{n}\|\{R_{k,i}x\} _{i \in I}\|\\
		&\leq \sum_{k=1}^{n}(\|\{T_{k,i}-R_{k,i}x\}_{i \in I}\|+\|\{T_{k,i}x\}_{i \in I}\|)\\
		&\leq (1+\sqrt{\lambda}) \|\sum_{k=1}^{n} \|\{T_{k,i}x\} _{i \in I}\|\\
		&\leq  (1+\sqrt{\lambda})(\sum_{k=1}^{n}\sqrt{B_k}) \|\langle  x, x\rangle_{\mathcal{A}}\|^{\frac{1}{2}}.
	\end{align*}
	Since, for any $x\in \mathcal{H}$, we have 
	$$\|L(\{\sum_{k=1}^{n}R_{k,i}x\} _{i \in I})\|=\|\{T_{p,i}x\} _{i \in I}\|.$$
	Then
	\begin{align*}
		\sqrt{A_p}\|\langle  x, x\rangle_{\mathcal{A}}\|^{\frac{1}{2}}&\leq \|\{T_{p,i}x\} _{i \in I}\|\\
		&= \|L(\{\sum_{k=1}^{n}R_{k,i}x\} _{i \in I})\|\\
		&\leq \|L\| \|\{\sum_{k=1}^{n}R_{k,i}x\} _{i \in I}\|.
	\end{align*}
	Hence
	$$\frac{\sqrt{A_p}}{\|L\|}\|\langle  x, x\rangle_{\mathcal{A}}\|^{\frac{1}{2}}\leq \|\{\sum_{k=1}^{n}R_{k,i}x\} _{i \in I}\|,\qquad  x\in \mathcal{H}.$$
	Therefore 
	$$\frac{\sqrt{A_p}}{\|L\|}\|\langle  x, x\rangle_{\mathcal{A}}\|^{\frac{1}{2}}\leq \|\{\sum_{k=1}^{n}R_{k,i}\} _{i \in I}\|\leq (1+\sqrt{\lambda})(\sum_{k=1}^{n}\sqrt{B_k}) \|\langle  x, x\rangle_{\mathcal{A}}\|^{\frac{1}{2}}.$$
	This give that $\{\sum_{k=1}^{n} R_{k,i}\} _{i \in I}$ is a  operator frame for $End_{\mathcal{A}}^{\ast}(\mathcal{H})$.
\end{proof}

\section{$K$-operator frame}

\begin{definition}
	Let $K\in End_{\mathcal{A}}^{\ast}(\mathcal{H})$. A family of adjointable operators $\{T_{i}\}_{i\in\mathbb{J}}$ on a Hilbert $\mathcal{A}$-module $\mathcal{H}$ over a unital $C^{\ast}$-algebra is said to be a $K$-operator frame for $End_{\mathcal{A}}^{\ast}(\mathcal{H})$, if there exists positive constants $A, B > 0$ such that 
	\begin{equation}\label{eq6}
		A\langle K^{\ast}x,K^{\ast}x\rangle\leq\sum_{i\in\mathbb{J}}\langle T_{i}x,T_{i}x\rangle\leq B\langle x,x\rangle, \forall x\in\mathcal{H}.
	\end{equation}
	The numbers $A$ and $B$ are called lower and upper bound of the $K$-operator frame, respectively. If $$A\langle K^{\ast}x,K^{\ast}x\rangle=\sum_{i\in\mathbb{J}}\langle T_{i}x,T_{i}x\rangle,$$ the $K$-operator frame is $A$-tight. If $A=1$, it is called a normalized tight $K$-operator frame or a Parseval $K$-operator frame. 
\end{definition}
\begin{Ex}
	Let $l^{\infty}$ be the set of all bounded complex-valued sequences. For any $u=\{u_{j}\}_{j\in\mathbb{N}}, v=\{v_{j}\}_{j\in\mathbb{N}}\in l^{\infty}$, we define$$
	uv=\{u_{j}v_{j}\}_{j\in\mathbb{N}}, u^{\ast}=\{\bar{u_{j}}\}_{j\in\mathbb{N}}, \|u\|=\sup_{j\in\mathbb{N}}|u_{j}|.$$
	Then $\mathcal{A}=\{l^{\infty}, \|.\|\}$ is a $\mathbb{C}^{\ast}$-algebra.
	
	Let $\mathcal{H}=C_{0}$ be the set of all sequences converging to zero. For any $u, v\in\mathcal{H}$ we define$$\langle u,v\rangle=uv^{\ast}=\{u_{j}\bar{u_{j}}\}_{j\in\mathbb{N}}.$$
	Then $\mathcal{H}$ is a Hilbert $\mathcal{A}$-module.
	
	Now let $\{e_{j}\}_{j\in\mathbb{N}}$ be the standard orthonormal basis of $\mathcal{H}$. For each $j\in\mathbb{N}$ define the adjointable operator$$T_{j}:\mathcal{H}\rightarrow\mathcal{H}, \;\;T_{j}x=\langle x,e_{j}\rangle e_{j},$$
	then for every $x\in\mathcal{H}$ we have$$\sum_{j\in\mathbb{N}}\langle T_{j}x,T_{j}x\rangle=\langle x,x\rangle.$$
	Fix $N\in\mathbb{N}^{\ast}$ and define$$K:\mathcal{H}\rightarrow\mathcal{H}, \;\;Ke_{j}=\begin{cases}
		je_{j}&\text{if $j\leq N$,}\\0&\text{if $j>N$.}
	\end{cases}$$
	It is easy to check that $K$ is adjointable and satisfies
	$$K^{\ast}e_{j}=\begin{cases}
		je_{j}&\text{if $j\leq N$,}\\0&\text{if $j>N$.}
	\end{cases}$$
	For any $x\in\mathcal{H}$ we have
	$$\dfrac{1}{N^{2}}\langle K^{\ast}x,K^{\ast}x\rangle\leq\sum_{j\in\mathbb{N}}\langle T_{j}x,T_{j}x\rangle=\langle x,x\rangle.$$
	This shows that $\{T_{j}\}_{j\in\mathbb{N}}$ is a $K$-operator frame with bounds $\dfrac{1}{N^{2}}, 1$.
\end{Ex}
One may ask for the class of operators $K$ which can guarantee the existence of $K$-operator frame for $End_{\mathcal{A}}^{\ast}(\mathcal{H})$. The following remark and proposition answer this query.
\begin{Rem}
	Every operator frame is a $K$-operator frame, for any $K\in End_{\mathcal{A}}^{\ast}(\mathcal{H})$: $K\neq0$. Indeed for any $K\in End_{\mathcal{A}}^{\ast}(\mathcal{H})$ the inequality 
	\begin{displaymath}
		\langle K^{\ast}x,K^{\ast}x\rangle\leq\|K\|^{2}\langle x,x\rangle, \forall x\in\mathcal{H} 
	\end{displaymath}
	holds. Now if $\{T_{i}\}_{i\in\mathbb{J}}$ is a operator frame with bounds $A$ and $B$ then
	\begin{displaymath}
		A\|K\|^{-2}\langle K^{\ast}x,K^{\ast}x\rangle\leq A\langle x,x\rangle\leq\sum_{i\in\mathbb{J}}\langle T_{i}x,T_{i}x\rangle\leq B\langle x,x\rangle, \forall x\in\mathcal{H}.
	\end{displaymath}
	Therefore $\{T_{i}\}_{i\in\mathbb{J}}$ is a $K$-operator frame with bounds $A\|K\|^{-2}$ and $B$.
\end{Rem}

\section{$K$-operator frame in tensor products of Hilbert $C^{\ast}$-modules}

Suppose that $\mathcal{A} , \mathcal{B}$ are $C^{\ast}$-algebras and we take $\mathcal{A}\otimes\mathcal{B}$ as the completion of $\mathcal{A}\otimes_{alg}\mathcal{B}$ with the spatial norm. $\mathcal{A}\otimes\mathcal{B}$ is the spatial tensor product of $\mathcal{A}$ and $\mathcal{B}$, also suppose that $\mathcal{H}$ is a Hilbert $\mathcal{A}$-module and $\mathcal{K}$ is a Hilbert $\mathcal{B}$-module. We want to define $\mathcal{H}\otimes\mathcal{K}$ as a Hilbert $(\mathcal{A}\otimes\mathcal{B})$-module.\\
Start by forming the algebraic tensor product $\mathcal{H}\otimes_{alg}\mathcal{K}$ of the vector spaces $\mathcal{H}$, $\mathcal{K}$ (over $\mathbb{C}$). 
This is a left module over $(\mathcal{A}\otimes_{alg}\mathcal{B})$ (the module action being given by 
$$(a\otimes b)(x\otimes y)=ax\otimes by (a\in\mathcal{A},b\in\mathcal{B},x\in\mathcal{H},y\in\mathcal{K})). $$
For $(x_{1},x_{2}\in\mathcal{H},y_{1},y_{2}\in\mathcal{K})$, we define 
$$ \langle x_{1}\otimes y_{1},x_{2}\otimes y_{2}\rangle_{\mathcal{A}\otimes\mathcal{B}}=\langle x_{1},x_{2}\rangle_{\mathcal{A}}\otimes\langle y_{1},y_{2}\rangle_{\mathcal{B}} . $$
We also know that for $ z=\sum_{i=1}^{n}x_{i}\otimes y_{i} $ in $\mathcal{H}\otimes_{alg}\mathcal{K}$, we have 
$$ \langle z,z\rangle_{\mathcal{A}\otimes\mathcal{B}}=\sum_{i,j}\langle x_{i},x_{j}\rangle_{\mathcal{A}}\otimes\langle y_{i},y_{j}\rangle_{\mathcal{B}}\geq0 $$  and $ \langle z,z\rangle_{\mathcal{A}\otimes\mathcal{B}}=0 $ iff $z=0$.
This extends by linearity to an $(\mathcal{A}\otimes_{alg}\mathcal{B})$-valued sesquilinear form on $\mathcal{H}\otimes_{alg}\mathcal{K}$, which makes $\mathcal{H}\otimes_{alg}\mathcal{K}$ into a semi-inner-product module over the pre-$\mathcal{C}^{\ast}$-algebra $(\mathcal{A}\otimes_{alg}\mathcal{B})$.
The semi-inner-product on $\mathcal{H}\otimes_{alg}\mathcal{K}$ is actually an inner product, see \cite{Lan}. Then $\mathcal{H}\otimes_{alg}\mathcal{K}$ is an inner-product module over the pre-$\mathcal{C}^{\ast}$-algebra $(\mathcal{A}\otimes_{alg}\mathcal{B})$, and we can perform the double completion discussed in chapter 1 of \cite{Lan} to conclude that the completion $\mathcal{H}\otimes\mathcal{K}$ of $\mathcal{H}\otimes_{alg}\mathcal{K}$ is a Hilbert $(\mathcal{A}\otimes\mathcal{B})$-module. We call $ \mathcal{H}\otimes\mathcal{K} $ the exterior tensor product of $\mathcal{H}$ and $\mathcal{K}$. 
With $\mathcal{H}$ ,$\mathcal{K}$ as above, we wish to investigate the adjointable operators on $ \mathcal{H}\otimes\mathcal{K} $. 
Suppose that $S\in End_{\mathcal{A}}^{\ast}( \mathcal{H})$ and $T\in End_{\mathcal{B}}^{\ast}(\mathcal{K})$. We define a linear operator $S\otimes T$ on $ \mathcal{H}\otimes\mathcal{K} $ by $S\otimes T(x\otimes y)=Sx\otimes Ty  (x\in\mathcal{H} ,y\in\mathcal{K})$. It is a routine verification that  is $S^{\ast}\otimes T^{\ast}$ is the adjoint of $S\otimes T$ , so in fact $S\otimes T\in End_{\mathcal{A\otimes B}}^{\ast}(\mathcal{H}\otimes\mathcal{K})$. For more details see \cite{Dav,Lan}. 
We note that if $a\in\mathcal{A}^{+}$ and $b\in\mathcal{B}^{+}$ , then $a\otimes b\in(\mathcal{A}\otimes\mathcal{B})^{+}$. 
Plainly if $a$ , $b$ are Hermitian elements of $\mathcal{A}$ and $a\geq b$ , then for every positive element $x$ of $\mathcal{B}$, we have $a\otimes x\geq b\otimes x$.

\begin{theorem}
	Let $\mathcal{H}$ and $\mathcal{K}$ be two Hilbert $C^{\ast}$-modules over unital $C^{\ast}$-algebras $\mathcal{A}$ and $\mathcal{B}$, respectively. Let $\{\Lambda_{i}\}_{i\in I}\subset End_{\mathcal{A}}^{\ast}(\mathcal{H})$ be a $K_{1}$-operator frame for $\mathcal{H}$ and $\{\Gamma_{j}\}_{j\in J}\subset End_{\mathcal{B}}^{\ast}(\mathcal{K})$ be a $K_{2}$-operator frame for $\mathcal{K}$ with frame operators $S_{\Lambda}$ and $S_{\Gamma}$ and operator frame bounds $(A,B)$ and $(C,D)$ respectively. Then $\{\Lambda_{i}\otimes\Gamma_{j}\}_{i\in I,j\in J}  $ is a $K_{1}\otimes K_{2}$-operator frame for Hibert $\mathcal{A}\otimes\mathcal{B}$-module $\mathcal{H}\otimes\mathcal{K}$ with frame operator $ S_{\Lambda}\otimes S_{\Gamma}$ and lower and upper operator frame bounds $AC$ and $BD$, respectively.
\end{theorem}

\begin{proof}
	By the definition of $K_{1}$-operator frame $\{\Lambda_{i}\}_{i\in I} $ and $K_{2}$-operator frame $\{\Gamma_{j}\}_{j\in J}$ we have 
	$$A\langle K_{1}^{\ast}x,K_{1}^{\ast}x\rangle_{\mathcal{A}}\leq\sum_{i\in I}\langle \Lambda_{i}x,\Lambda_{i}x\rangle_{\mathcal{A}}\leq B\langle x,x\rangle_{\mathcal{A}} , \forall x\in\mathcal{H}.$$
	$$C\langle K_{2}^{\ast}y,K_{2}^{\ast}y\rangle_{\mathcal{B}}\leq\sum_{j\in J}\langle \Gamma_{j}y,\Gamma_{j}y\rangle_{\mathcal{B}}\leq D\langle y,y\rangle_{\mathcal{B}} , \forall y\in\mathcal{K}.$$
	Therefore
	\begin{align*}
		(A\langle K_{1}^{\ast}x,K_{1}^{\ast}x\rangle_{\mathcal{A}})\otimes (C\langle K_{2}^{\ast}y,K_{2}^{\ast}y\rangle_{\mathcal{B}})&\leq\sum_{i\in I}\langle \Lambda_{i}x,\Lambda_{i}x\rangle_{\mathcal{A}}\otimes\sum_{j\in J}\langle \Gamma_{j}y,\Gamma_{j}y\rangle_{\mathcal{B}}\\
		&\leq (B\langle x,x\rangle_{\mathcal{A}})\otimes (D\langle y,y\rangle_{\mathcal{B}}) , \forall x\in\mathcal{H} ,\forall y\in\mathcal{K}.
	\end{align*}
	Then
	\begin{align*}
			AC(\langle K_{1}^{\ast}x,K_{1}^{\ast}x\rangle_{\mathcal{A}}\otimes\langle K_{2}^{\ast}y,K_{2}^{\ast}y\rangle_{\mathcal{B}}) &\leq\sum_{i\in I,j\in J}\langle \Lambda_{i}x,\Lambda_{i}x\rangle_{\mathcal{A}}\otimes\langle \Gamma_{j}y,\Gamma_{j}y\rangle_{\mathcal{B}}\\
		&\leq BD(\langle x,x\rangle_{\mathcal{A}}\otimes\langle y,y\rangle_{\mathcal{B}}) , \forall x\in\mathcal{H} ,\forall y\in\mathcal{K}.
	\end{align*}
	Consequently we have
	\begin{align*}
		AC\langle K_{1}^{\ast}x\otimes K_{2}^{\ast}y,K_{1}^{\ast}x\otimes K_{2}^{\ast}y\rangle_{\mathcal{A}\otimes\mathcal{B}} &\leq\sum_{i\in I,j\in J}\langle\Lambda_{i}x\otimes\Gamma_{j}y,\Lambda_{i}x\otimes\Gamma_{j}y\rangle_{\mathcal{A}\otimes\mathcal{B}} \\
		&\leq BD\langle x\otimes y,x\otimes y\rangle_{\mathcal{A}\otimes\mathcal{B}}, \forall x\in\mathcal{H}, \forall y\in\mathcal{K}.
	\end{align*}
	Then for all $x\otimes y$ in $\mathcal{H\otimes K}$ we have
	\begin{multline*}
		AC\langle(K_{1}\otimes K_{2})^{\ast}(x\otimes y),(K_{1}\otimes K_{2})^{\ast}(x\otimes y)\rangle_{\mathcal{A}\otimes\mathcal{B}} \\\leq\sum_{i\in I,j\in J}\langle(\Lambda_{i}\otimes\Gamma_{j})(x\otimes y),(\Lambda_{i}\otimes\Gamma_{j})(x\otimes y)\rangle_{\mathcal{A}\otimes\mathcal{B}} \\
		\leq BD\langle x\otimes y,x\otimes y\rangle_{\mathcal{A}\otimes\mathcal{B}}.
	\end{multline*}
	The last inequality is satisfied for every finite sum of elements in $\mathcal{H}\otimes_{alg}\mathcal{K}$ and then it's satisfied for all $z\in\mathcal{H\otimes K}$. It shows that $\{\Lambda_{i}\otimes\Gamma_{j}\}_{i\in I,j\in J}  $ is a $K_{1}\otimes K_{2}$-operator frame for Hilbert $\mathcal{A}\otimes\mathcal{B}$-module $\mathcal{H}\otimes\mathcal{K}$ with lower and upper operator frame bounds $AC$ and $ BD $, respectively.\\
	By the definition of frame operator $S_{\Lambda}$ and $S_{\Gamma}$, we have
	$$S_{\Lambda}x=\sum_{i\in I}\Lambda_{i}^{\ast}\Lambda_{i}x, \forall x\in\mathcal{H},$$
	and
	$$S_{\Gamma}y=\sum_{j\in J}\Gamma_{j}^{\ast}\Gamma_{j}y, \forall y\in\mathcal{K}.$$
	Therefore
	$$
	\aligned
	(S_{\Lambda}\otimes S_{\Gamma})(x\otimes y)&=S_{\Lambda}x\otimes S_{\Gamma}y\\
	&=\sum_{i\in I}\Lambda_{i}^{\ast}\Lambda_{i}x\otimes\sum_{j\in J}\Gamma_{j}^{\ast}\Gamma_{j}y\\
	&=\sum_{i\in I,j\in J}\Lambda_{i}^{\ast}\Lambda_{i}x\otimes\Gamma_{j}^{\ast}\Gamma_{j}y\\
	&=\sum_{i\in I,j\in J}(\Lambda_{i}^{\ast}\otimes\Gamma_{j}^{\ast})(\Lambda_{i}x\otimes\Gamma_{j}y)\\
	&=\sum_{i\in I,j\in J}(\Lambda_{i}^{\ast}\otimes\Gamma_{j}^{\ast})(\Lambda_{i}\otimes\Gamma_{j})(x\otimes y)\\
	&=\sum_{i\in I,j\in J}(\Lambda_{i}\otimes\Gamma_{j})^{\ast})(\Lambda_{i}\otimes\Gamma_{j})(x\otimes y).
	\endaligned
	$$
	Now by the uniqueness of frame operator, the last expression is equal to $S_{\Lambda\otimes\Gamma}(x\otimes y)$. Consequently we have $ (S_{\Lambda}\otimes S_{\Gamma})(x\otimes y)=S_{\Lambda\otimes\Gamma}(x\otimes y)$. The last equality is satisfied for every finite sum of elements in $\mathcal{H}\otimes_{alg}\mathcal{K}$ and then it's satisfied for all $z\in\mathcal{H\otimes K}$. It shows that $ (S_{\Lambda}\otimes S_{\Gamma})(z)=S_{\Lambda\otimes\Gamma}(z)$. So $S_{\Lambda\otimes\Gamma}=S_{\Lambda}\otimes S_{\Gamma}$.
\end{proof}

\begin{theorem}
	Assume that $Q\in End_{\mathcal{A}}^{\ast}(\mathcal{H})$ is invertible and $\{\Lambda_{i}\}_{i\in I}\subset End_{\mathcal{A\otimes B}}^{\ast}(\mathcal{H\otimes K})$ is a $K$-operator frame for $\mathcal{H\otimes K}$ with lower and upper  operator frame bounds $A$ and $B$ respectively and frame operator $S$. If $K$ commute with $Q\otimes I$, then $\{\Lambda_{i}(Q^{\ast}\otimes I)\}_{i\in I} $ is a $K$-operator frame for $\mathcal{H\otimes K}$ with lower and upper operator frame bounds $\|Q^{\ast-1}\|^{-2}A$ and $\|Q\|^{2}B$ respectively and frame operator $(Q\otimes I)S(Q^{\ast}\otimes I)$.
\end{theorem}

\begin{proof}
	Since $Q\in End_{\mathcal{A}}^{\ast}(\mathcal{H})$, $Q\otimes I\in End_{\mathcal{A\otimes B}}^{\ast}(\mathcal{H\otimes K})$ with inverse $Q^{-1}\otimes I$. It is obvious that the adjoint of $Q\otimes I$ is $Q^{\ast}\otimes I$. An easy calculation shows that for every elementary tensor $x\otimes y$,
	$$
	\aligned
	\|(Q\otimes I)(x\otimes y)\|^{2}&=\|Q(x)\otimes y\|^{2}\\
	&=\|Q(x)\|^{2}\|y\|^{2}\\
	&\leq\|Q\|^{2}\|x\|^{2}\|y\|^{2}\\
	&=\|Q\|^{2}\|x\otimes y\|^{2}.
	\endaligned
	$$
	So $Q\otimes I$ is bounded, and therefore it can be extended to $\mathcal{H\otimes K}$. Similarly for $Q^{\ast}\otimes I$, hence $Q\otimes I$ is $\mathcal{A\otimes B}$-linear, adjointable with adjoint $Q^{\ast}\otimes I$. Hence for every $z\in\mathcal{H\otimes K}$ we have $$\|Q^{\ast-1}\|^{-1}.|z|\leq|(Q^{\ast}\otimes I)z|\leq\|Q\|.|z|.$$
	By the definition of $K$-operator frames we have 
	$$A\langle K^{\ast}z,K^{\ast}z\rangle_{\mathcal{A\otimes B}}\leq\sum_{i\in I}\langle \Lambda_{i}z,\Lambda_{i}z\rangle_{\mathcal{A\otimes B}}\leq B\langle z,z\rangle_{\mathcal{A\otimes B}}.$$
	Then 
	$$
	\aligned
	A\langle K^{\ast}(Q^{\ast}\otimes I)z,K^{\ast}(Q^{\ast}\otimes I)z\rangle_{\mathcal{A\otimes B}} &\leq\sum_{i\in I}\langle \Lambda_{i}(Q^{\ast}\otimes I)z,\Lambda_{i}(Q^{\ast}\otimes I)z\rangle_{\mathcal{A\otimes B}}\\
	&\leq B\langle (Q^{\ast}\otimes I)z,(Q^{\ast}\otimes I)z\rangle_{\mathcal{A\otimes B}}\\
	&\leq \|Q\|^{2}B\langle z,z\rangle_{\mathcal{A\otimes B}}.
	\endaligned
	$$
	Or
	$$
	\aligned
	A\langle K^{\ast}(Q^{\ast}\otimes I)z,K^{\ast}(Q^{\ast}\otimes I)z\rangle_{\mathcal{A\otimes B}} &=A\langle (Q^{\ast}\otimes I)K^{\ast}z,(Q^{\ast}\otimes I)K^{\ast}z\rangle_{\mathcal{A\otimes B}}\\
	&\geq \|Q^{\ast-1}\|^{-2}A\langle K^{\ast}z,K^{\ast}z\rangle_{\mathcal{A\otimes B}}.
	\endaligned
	$$
	So we have$$\|Q^{\ast-1}\|^{-2}A\langle K^{\ast}z,K^{\ast}z\rangle_{\mathcal{A\otimes B}}\leq\sum_{i\in I}\langle \Lambda_{i}(Q^{\ast}\otimes I)z,\Lambda_{i}(Q^{\ast}\otimes I)z\rangle_{\mathcal{A\otimes B}}\leq\|Q\|^{2}B\langle z,z\rangle_{\mathcal{A\otimes B}}.$$
	Now
	$$
	\aligned
	(Q\otimes I)S(Q^{\ast}\otimes I)&=(Q\otimes I)(\sum_{i\in I}\Lambda_{i}^{\ast}\Lambda_{i})(Q^{\ast}\otimes I)\\
	&=\sum_{i\in I}(Q\otimes I)\Lambda_{i}^{\ast}\Lambda_{i}(Q^{\ast}\otimes I)\\
	&=\sum_{i\in I}(\Lambda_{i}(Q^{\ast}\otimes I))^{\ast}\Lambda_{i}(Q^{\ast}\otimes I).
	\endaligned
	$$
	Which completes the proof.
\end{proof}

\begin{theorem}
	Assume that $Q\in End_{\mathcal{B}}^{\ast}(\mathcal{K})$ is invertible and $\{\Lambda_{i}\}_{i\in I}\subset End_{\mathcal{A\otimes B}}^{\ast}(\mathcal{H\otimes K})$ is a $K$-operator frame for $\mathcal{H\otimes K}$ with lower and upper  operator frame bounds $A$ and $B$ respectively and frame operator $S$. If $K$ commute with $I\otimes Q$, then $\{\Lambda_{i}(I\otimes Q^{\ast})\}_{i\in I} $ is a $K$-operator frame for $\mathcal{H\otimes K}$ with lower and upper operator frame bounds $\|Q^{\ast-1}\|^{-2}A$ and $\|Q\|^{2}B$ respectively and frame operator $(I\otimes Q)S(I\otimes Q^{\ast})$.
\end{theorem}

\begin{proof}
	Similar to the proof of the previous theorem.
\end{proof}

\section{$K$-operator frame in Hilbert $C^{\ast}$-modules with Different $C^{\ast}$-algebras}
Studying operator frame in Hilbert $C^{\ast}$-modules with different $C^{\ast}$-algebras is interesting and important. In the following theorem we study this situation.
\begin{theorem}
	Let $(\mathcal{H},\mathcal{A},\langle.,.\rangle_{\mathcal{A}})$ and $(\mathcal{H},\mathcal{B},\langle.,.\rangle_{\mathcal{B}})$ be two Hilbert $\mathcal{C^{\ast}}$-modules and let $\varphi :\mathcal{A}\longrightarrow \mathcal{B}$ be a $\ast$-homomorphism and $\theta$ be a map on $\mathcal{H}$ such that $\langle \theta x,\theta y\rangle_{\mathcal{B}}=\varphi(\langle x, y\rangle_{\mathcal{A}})$ for all $x,y\in\mathcal{H}$. Also, suppose that $\{\Lambda_{i}\}_{i\in I}\subset End_{\mathcal{A}}^{\ast}(\mathcal{H})$ is a $K$-operator frame for $(\mathcal{H},\mathcal{A},\langle.,.\rangle_{\mathcal{A}})$ with frame operator $S_{\mathcal{A}} $ and lower and upper operator frame bounds $A$, $B$  respectively. If $\theta$ is surjective, $\theta K^{\ast}=K^{\ast}\theta$, $\theta\Lambda_{i}=\Lambda_{i}\theta$ and $\theta\Lambda_{i}^{\ast}=\Lambda_{i}^{\ast}\theta$ for each $i$ in $I$, then $\{\Lambda_{i}\}_{i\in I}$ is a $K$-operator frame for $(\mathcal{H},\mathcal{B},\langle.,.\rangle_{\mathcal{B}})$ with frame operator $S_{\mathcal{B}} $ and lower and upper operator frame bounds $A$, $B$ respectively, and $\langle S_{\mathcal{B}} \theta x,\theta y\rangle_{\mathcal{B}}=\varphi(\langle S_{\mathcal{A}}x, y\rangle_{\mathcal{A}})$.
\end{theorem}

\begin{proof} Let $y\in\mathcal{H}$ then there exists $x\in\mathcal{H}$ such that $\theta x=y$ ($\theta$ is surjective). By the definition of $K$-operator frames, we have
	$$A\langle K^{\ast}x,K^{\ast}x\rangle_{\mathcal{A}}\leq\sum_{i\in I}\langle \Lambda_{i}x,\Lambda_{i}x\rangle_{\mathcal{A}}\leq B\langle x,x\rangle_{\mathcal{A}}.$$
	We have
	$$\varphi(A\langle K^{\ast}x,K^{\ast}x\rangle_{\mathcal{A}})\leq\varphi(\sum_{i\in I}\langle \Lambda_{i}x,\Lambda_{i}x\rangle_{\mathcal{A}})\leq\varphi( B\langle x,x\rangle_{\mathcal{A}}).$$
	By the definition of $\ast$-homomorphism, we have
	$$A\varphi(\langle K^{\ast}x,K^{\ast}x\rangle_{\mathcal{A}}) \leq\sum_{i\in I}\varphi(\langle \Lambda_{i}x,\Lambda_{i}x\rangle_{\mathcal{A}})\leq B\varphi(\langle x,x\rangle_{\mathcal{A}}).$$
	By the relation betwen $\theta$ and $\varphi$, we get
	$$A\langle \theta K^{\ast}x,\theta K^{\ast}x\rangle_{\mathcal{B}} \leq\sum_{i\in I}\langle \theta\Lambda_{i}x,\theta\Lambda_{i}x\rangle_{\mathcal{B}}\leq B\langle\theta x,\theta x\rangle_{\mathcal{B}}.$$
	By the relation betwen $\theta$, $K^{\ast}$ and $\Lambda_{i}$, we have
	$$A\langle  K^{\ast}\theta x, K^{\ast}\theta x\rangle_{\mathcal{B}} \leq\sum_{i\in I}\langle \Lambda_{i}\theta x,\Lambda_{i}\theta x\rangle_{\mathcal{B}}\leq B\langle\theta x,\theta x\rangle_{\mathcal{B}}.$$
	Then
	$$A\langle  K^{\ast}y, K^{\ast}y\rangle_{\mathcal{B}} \leq\sum_{i\in I}\langle \Lambda_{i}y,\Lambda_{i}y\rangle_{\mathcal{B}}\leq B\langle y,y\rangle_{\mathcal{B}}, \forall y\in\mathcal{H}.$$
	On the other hand, we have
	$$
	\aligned
	\varphi(\langle S_{\mathcal{A}}x, y\rangle_{\mathcal{A}})&=\varphi(\langle\sum_{i\in I}\Lambda_{i}^{\ast}\Lambda_{i}x,y\rangle_{\mathcal{A}})\\
	&=\sum_{i\in I}\varphi(\langle\Lambda_{i}x,\Lambda_{i}y\rangle_{\mathcal{A}})\\
	&=\sum_{i\in I}\langle\theta\Lambda_{i}x,\theta\Lambda_{i}y\rangle_{\mathcal{B}}
	\\
	&=\sum_{i\in I}\langle\Lambda_{i}\theta x,\Lambda_{i}\theta y\rangle_{\mathcal{B}}\\
	&=\langle\sum_{i\in I}\Lambda_{i}^{\ast}\Lambda_{i}\theta x,\theta y\rangle_{\mathcal{B}}\\
	&=\langle S_{\mathcal{B}}\theta x,\theta y\rangle_{\mathcal{B}}.
	\endaligned
	$$
	Which completes the proof.
\end{proof}

\section{Duals of $K$-operator frame}
In the following we define the Dual $K$-operator frame and we give some properties
\begin{definition}
	
	Let $K \in End_{\mathcal{A}}^{\ast}(\mathcal{H})$ and $\{ \Lambda_{i} \in End_{\mathcal{A}}^{\ast}(\mathcal{H}), i\in I \}$ be a K-operator frame for the Hilbert $\mathcal{A}$-module $H$. An operator Bessel sequences $\{ \Gamma_{i} \in End_{\mathcal{A}}^{\ast}(\mathcal{H}), i\in I \}$ is called a K-duals operator frame for $\{ \Lambda_{i}\}_{i \in I}$ if 
	$Kf=\sum_{i\in I}\Lambda_{i}^{\ast}\Gamma_{i}f$ for all $f\in H$. 	
\end{definition}
\begin{Ex}
	
	Let $K {\in End_\mathcal{A}}^{\ast}(\mathcal{H})$ be a surjective operator and $\{\Lambda_{i} \in End_{\mathcal{A}}^{\ast}(\mathcal{H}), i\in I\}$ be a K-operator frame for $\mathcal{H}$ with K-frame operator $S$, then $S$ is invertible.\\
	For all $f \in \mathcal{H}$, we have :
	$$Sf=\sum_{i\in I}\Lambda_{i}^{\ast}\Lambda_{i}f.$$
	So	$$Kf=\sum_{i\in I}\Lambda_{i}^{\ast}\Lambda_{i}S^{-1}Kf.$$
	Then the sequence  $\{ \Lambda_{i}S^{-1}K \in End_{\mathcal{A}}^{\ast}(\mathcal{H}), i\in I \}$ is a dual K-operator frame of  $\{ \Lambda_{i} \in End_{\mathcal{A}}^{\ast}(\mathcal{H}), i\in I \}$ 
\end{Ex}
\begin{theorem}
	
	Let $\{\Lambda_{i}\}_{i\in I} $ and $\{\Gamma_{j}\}_{j\in J} $ are K-operator frame and L-operator frame respectively in $\mathcal{H}$ and $\mathcal{K}$, with duals $\{\tilde{\Lambda_{i}}\}_{i\in I}$ and  $\{\tilde{\Gamma_{j}}\}_{j\in J} $ respectively, then $\{\tilde{\Lambda_{i}}\otimes\tilde{\Gamma_{j}}\}_{i,j\in I,J}$ is a dual of $\{\Lambda_{i}\otimes\Gamma_{j}\}_{i,j\in I,J}$.
\end{theorem}

\begin{proof}
	By definition, $\forall x \in \mathcal{H}$ and $\forall y \in \mathcal{K}$ we have:
	$$
		\sum_{i\in I}\Lambda_{i}^{\ast}\tilde{\Lambda_{i}}x=Kx  
	\text{ and }
		\sum_{j\in J}\Gamma_{j}^{\ast}\tilde{\Gamma_{j}}y=Ly  
 	$$
	then : 
	\begin{equation*}
		(K\otimes L)(x\otimes y)= Kx\otimes Ly = \sum_{i\in I}\Lambda_{i}^{\ast}\tilde{\Lambda_{i}}x\otimes\sum_{j\in J}\Gamma_{j}^{\ast}\tilde{\Gamma_{j}}y
	\end{equation*}
	\begin{equation*}
		\sum_{i\in I}\Lambda_{i}^{\ast}\tilde{\Lambda_{i}}x\otimes\sum_{j\in J}\Gamma_{j}^{\ast}\tilde{\Gamma_{j}}y=\sum_{i,j\in I,J}\Lambda_{i}^{\ast}\tilde{\Lambda_{i}}x\otimes\Gamma_{j}^{\ast}\tilde{\Gamma_{j}}y
	\end{equation*}
	
	\begin{equation*}	
		\sum_{i\in I}\Lambda_{i}^{\ast}\tilde{\Lambda_{i}}x\otimes\sum_{j\in J}\Gamma_{j}^{\ast}\tilde{\Gamma_{j}}y=\sum_{i,j\in I,J}(\Lambda_{i}^{\ast}\otimes\Gamma_{j}^{\ast}).(\tilde{\Lambda_{i}}x\otimes\tilde{\Gamma_{j}}y)
	\end{equation*}
	\begin{equation*}
		\sum_{i\in I}\Lambda_{i}^{\ast}\tilde{\Lambda_{i}}x\otimes\sum_{j\in J}\Gamma_{j}^{\ast}\tilde{\Gamma_{j}}y=\sum_{i,j\in I,J}(\Lambda_{i}\otimes\Gamma_{j})^{\ast}.(\tilde{\Lambda_{i}}\otimes\tilde{\Gamma_{j}}).(x\otimes y)
	\end{equation*}
	then $\{\tilde{\Lambda_{i}}\otimes\tilde{\Gamma_{j}}\}_{i,j \in I,J}$ is a dual of $\{\Lambda_{i}\otimes\Gamma_{j}\}_{i,j \in I,J}$.
\end{proof}

\begin{corollary}
	
	Let $\{\Lambda_{i,j}\}_{0\leq i\leq n; j\in J}$ be a family of $K_{i}$-operator frames, such $ 0 \leq i \leq n $ and $\{\tilde{\Lambda}_{i,j}\}_{0\leq i\leq n; j\in J}$ their dual, then $\{\tilde{\Lambda}_{0,j}\otimes \tilde{\Lambda}_{1,j}\otimes......\otimes\tilde{\Lambda}_{n,j}\}_{j\in J}$ is a dual of $\{\Lambda_{0,j}\otimes \Lambda_{1,j}\otimes......\otimes\Lambda_{n,j}\}_{j\in J}$.
\end{corollary}

\section{$\ast$-operator frame}

\begin{definition}
	A family of adjointable operators $\{T_{i}\}_{i\in I}$ on a Hilbert $\mathcal{A}$-module $\mathcal{H}$ over a unital $C^{\ast}$-algebra is said to be an $\ast$-operator frame for $End_{\mathcal{A}}^{\ast}(\mathcal{H})$, if there exists two strictly nonzero elements $A$ and $B$ in $\mathcal{A}$ such that 
	\begin{equation}\label{eqq33}
		A\langle x,x\rangle A^{\ast}\leq\sum_{i\in I}\langle T_{i}x,T_{i}x\rangle\leq B\langle x,x\rangle B^{\ast}, \forall x\in\mathcal{H}.
	\end{equation}
	The elements $A$ and $B$ are called lower and upper bounds of the $\ast$-operator frame, respectively. If $A=B=\lambda$, the $\ast$-operator frame is $\lambda$-tight. If $A = B = 1_{\mathcal{A}}$, it is called a normalized tight $\ast$-operator frame or a Parseval $\ast$-operator frame. If only upper inequality of \eqref{eqq33} hold, then $\{T_{i}\}_{i\in i}$ is called an $\ast$-operator Bessel sequence for $End_{\mathcal{A}}^{\ast}(\mathcal{H})$.
\end{definition}
We mentioned that the set of all of operator frames for $End_{\mathcal{A}}^{\ast}(\mathcal{H})$ can be considered
as a subset of $\ast$-operator frame. To illustrate this, let $\{T_{j}\}_{i\in I}$ be an operator frame for Hilbert $\mathcal{A}$-module $\mathcal{H}$
with operator frame real bounds $A$ and $B$. Note that for $x\in\mathcal{H}$,
\begin{equation*}
	(\sqrt{A})1_{\mathcal{A}}\langle x,x\rangle_{\mathcal{A}}(\sqrt{A})1_{\mathcal{A}}\leq\sum_{i\in I}\langle T_{i}x,T_{i}x\rangle\leq(\sqrt{B})1_{\mathcal{A}}\langle x,x\rangle_{\mathcal{A}}(\sqrt{B})1_{\mathcal{A}}.
\end{equation*}
Therefore, every operator frame for $End_{\mathcal{A}}^{\ast}(\mathcal{H})$ with real bounds $A$ and $B$ is an $\ast$-operator frame for $End_{\mathcal{A}}^{\ast}(\mathcal{H})$ with $\mathcal{A}$-valued $\ast$-operator frame bounds $(\sqrt{A})1_{\mathcal{A}}$ and $(\sqrt{B})1_{\mathcal{B}}$.
\begin{Ex}
	Let $\mathcal{A}$ be a Hilbert $C^{\ast}$-module over itself with the inner product $\langle a,b\rangle=ab^{\ast}$.
	Let $\{x_{i}\}_{i\in I}$ be an $\ast$-frame for $\mathcal{A}$ with bounds $A$ and $B$, respectively. For each $i\in I$, we define $T_{i}:\mathcal{A}\to\mathcal{A}$ by $T_{i}x=\langle x,x_{i}\rangle,\;\; \forall x\in\mathcal{A}$. $T_{i}$ is adjointable and $T_{i}^{\ast}a=ax_{i}$ for each $a\in\mathcal{A}$. And we have 
	\begin{equation*}
		A\langle x,x\rangle A^{\ast}\leq\sum_{i\in I}\langle x,x_{i}\rangle\langle x_{i},x\rangle\leq B\langle x,x\rangle B^{\ast}, \forall x\in\mathcal{A}.
	\end{equation*}
	Then
	\begin{equation*}
		A\langle x,x\rangle A^{\ast}\leq\sum_{i\in I}\langle T_{i}x,T_{i}x\rangle\leq B\langle x,x\rangle B^{\ast}, \forall x\in\mathcal{A}.
	\end{equation*}
	So $\{T_{i}\}_{i\in I}$ is an $\ast$-operator frame in $\mathcal{A}$ with bounds $A$ and $B$, respectively.
\end{Ex}
\begin{Ex}
	Let $\{W_{i}\}_{i\in\mathbb{J}}$ be a $\ast$-frame of submodules with respect to $ \{v_{i}\}_{i\in\mathbb{J}} $ for $\mathcal{H}$. Put $T_{i}=v_{i}\pi_{W_{i}}, \forall i\in\mathbb{J}$, then we get a sequence of operators $\{T_{i}\}_{i\in\mathbb{J}}$. Then there exist  $A, B\in\mathcal{A}$ such that:
	\begin{displaymath}
		A\langle x,x\rangle A^{\ast}\leq\sum_{i\in\mathbb{J}}v_{i}^{2}\langle \pi_{W_{i}}x,\pi_{W_{i}}x\rangle\leq B\langle x,x\rangle B^{\ast}, \forall x\in\mathcal{H}.
	\end{displaymath}
	So we have:
	\begin{displaymath}
		A\langle x,x\rangle A^{\ast}\leq\sum_{i\in\mathbb{J}}\langle T_{i}x,T_{i}x\rangle\leq B\langle x,x\rangle B^{\ast}, \forall x\in\mathcal{H}.
	\end{displaymath}
	Thus, the sequence $\{T_{i}\}_{i\in\mathbb{J}}$ becomes a $\ast$-operator frame for $\mathcal{H}$.
\end{Ex}
With this example a $\ast$-frame of submodules can be viewed as a special case of $\ast$-operator frames.
\begin{Rem}
	The examples $3.3$ and $3.4$ in \cite{A} are examples of $\ast$-operator frame.
\end{Rem}

\section{$\ast$-$K$-operator frame}

Now we define the $\ast$-$K$-operator frame for $End_{\mathcal{A}}^{\ast}(\mathcal{H})$.
\begin{definition}
	Let $K\in End_{\mathcal{A}}^{\ast}(\mathcal{H})$. A family of adjointable operators $\{T_{i}\}_{i\in I}$, on a Hilbert $\mathcal{A}$-module $\mathcal{H}$ over a unital $C^{\ast}$-algebra, is said an $\ast$-$K$-operator frame for $End_{\mathcal{A}}^{\ast}(\mathcal{H})$, if there exists two nonzero elements $A$ and $B$ in $\mathcal{A}$ such that 
	\begin{equation}\label{eqqq33}
		A\langle K^{\ast}x,K^{\ast}x\rangle A^{\ast}\leq\sum_{i\in I}\langle T_{i}x,T_{i}x\rangle\leq B\langle x,x\rangle B^{\ast}, \forall x\in\mathcal{H}.
	\end{equation}
	The elements $A$ and $B$ are called lower and upper bounds of the $\ast$-$K$-operator frame, respectively. If $$A\langle K^{\ast}x,K^{\ast}x\rangle_{\mathcal{A}}A^{\ast}=\sum_{i\in\mathbb{J}}\langle T_{i}x,T_{i}x\rangle_{\mathcal{A}},$$ the $\ast$-$K$-operator frame is an $A$-tight. If $A=1$, it is called a normalized tight $\ast$-$K$-operator frame or a Parseval $\ast$-$K$-operator frame. 
\end{definition}
\begin{Ex}
	Let $l^{\infty}$ be the set of all bounded complex-valued sequences. For any $u=\{u_{j}\}_{j\in\mathbb{N}}, v=\{v_{j}\}_{j\in\mathbb{N}}\in l^{\infty}$, we define
	\begin{equation*}
		uv=\{u_{j}v_{j}\}_{j\in\mathbb{N}}, u^{\ast}=\{\bar{u_{j}}\}_{j\in\mathbb{N}}, \|u\|=\sup_{j\in\mathbb{N}}|u_{j}|.
	\end{equation*}
	Then $\mathcal{A}=\{l^{\infty}, \|.\|\}$ is a $\mathbb{C}^{\ast}$-algebra.
	
	Let $\mathcal{H}=C_{0}$ be the set of all null sequences. For any $u, v\in\mathcal{H}$ we define
	\begin{equation*}
		\langle u,v\rangle=uv^{\ast}=\{u_{j}\bar{u_{j}}\}_{j\in\mathbb{N}}.
	\end{equation*}
	Then $\mathcal{H}$ is a Hilbert $\mathcal{A}$-module.
	\\
	Define $f_{j}=\{f_{i}^{j}\}_{i\in\mathbb{N}^{\ast}}$ by $f_{i}^{j}=\frac{1}{2}+\frac{1}{i}$ if $i=j$ and $f_{i}^{j}=0$ if $i\neq j$ $\forall j\in\mathbb{N}^{\ast}$.
	\\
	Now, define the adjointable operator $T_{j}: \mathcal{H}\to\mathcal{H},\;\; T_{j}\{(x_{i})_{i}\}=(x_{i}f_{i}^{j})_{i}$.
	\\
	Then for every $x\in\mathcal{H}$ we have
	\begin{equation*}
		\sum_{j\in\mathbb{N}}\langle T_{j}x,T_{j}x\rangle=\{\dfrac{1}{2}+\dfrac{1}{i}\}_{i\in\mathbb{N}^{\ast}}\langle x,x\rangle\{\dfrac{1}{2}+\dfrac{1}{i}\}_{i\in\mathbb{N}^{\ast}}.
	\end{equation*} 
	So $\{T_{j}\}_{j}$ is a $\{\frac{1}{2}+\frac{1}{i}\}_{i\in\mathbb{N}^{\ast}}$-tight $\ast$-operator frame.
	\\
	Let $K:\mathcal{H}\to\mathcal{H}$ defined by $Kx=\{\frac{x_{i}}{i}\}_{i\in\mathbb{N}^{\ast}}$.
	\\
	Then for every $x\in\mathcal{H}$ we have 
	\begin{equation*}
		\langle K^{\ast}x,K^{\ast}x\rangle_{\mathcal{A}}\leq\sum_{j\in\mathbb{N}}\langle T_{j}x,T_{j}x\rangle=\{\frac{1}{2}+\frac{1}{i}\}_{i\in\mathbb{N}^{\ast}}\langle x,x\rangle\{\frac{1}{2}+\frac{1}{i}\}_{i\in\mathbb{N}^{\ast}}.
	\end{equation*}
	This shows that $\{T_{j}\}_{j\in\mathbb{N}}$ is an $\ast$-$K$-operator frame with bounds $1, \{\frac{1}{2}+\frac{1}{i}\}_{i\in\mathbb{N}^{\ast}}$.
\end{Ex}
\begin{Rem} \label{rem3.2}
	\begin{enumerate}
		\item Every $\ast$-operator frame for $End_{\mathcal{A}}^{\ast}(\mathcal{H})$ is an $\ast$-K-operator frame, for any $K\in End_{\mathcal{A}}^{\ast}(\mathcal{H})$: $K\neq0$.
		\item If $K\in End_{\mathcal{A}}^{\ast}(\mathcal{H})$ is a surjective operator, then every $\ast$-K-operator frame for $End_{\mathcal{A}}^{\ast}(\mathcal{H})$ is an $\ast$-operator frame.
	\end{enumerate}
\end{Rem}

\noindent
{\bf Acknowledgements.}
We would like to thank the anonymous referees for their careful reading and we appreciate their valuable comments.

\end{document}